\documentclass[reqno,12pt]{amsart}

\usepackage[margin=1in]{geometry}
\usepackage{xspace}
\usepackage[utf8]{inputenc}
\usepackage{graphicx}
\usepackage{amssymb}
\usepackage{mathrsfs}
\usepackage[active]{srcltx}
\usepackage{url}
\usepackage{hyperref}
\usepackage{amsmath,amstext,amsxtra,amsgen,amsbsy,amsopn,amscd,amsthm,amsfonts}
\usepackage{latexsym}
\usepackage{dsfont}
\usepackage{enumitem}

\newtheorem{theorem}{Theorem}

\newtheorem{proposition}{Proposition}
\newtheorem{lemma}[proposition]{Lemma}
\newtheorem{corollary}[proposition]{Corollary}

\theoremstyle{definition}

\newtheorem{definition}[proposition]{Definition}

\theoremstyle{remark}

\newtheorem{remark}[proposition]{Remark}

\numberwithin{equation}{section}
\numberwithin{proposition}{section}

\newcommand\R{{\ensuremath {\mathbb R} }}
\newcommand\C{{\ensuremath {\mathbb C} }}
\newcommand\N{{\ensuremath {\mathbb N} }}
\newcommand\Z{{\ensuremath {\mathbb Z} }}

\newcommand\T{{\ensuremath {\mathbb T} }}

\renewcommand\phi{\varphi}

\renewcommand\le{\leqslant}
\renewcommand\ge{\geqslant}
\renewcommand\epsilon{\varepsilon}
\renewcommand\hat{\widehat}
\renewcommand\tilde{\widetilde}
\renewcommand\bar{\overline}

\newcommand{\cD}{\mathcal{D}}

\newcommand\ii{{\ensuremath {\infty}}}

\newcommand{\norm}[1]{ \left| \! \left| #1 \right| \! \right| }

\newcommand{\cM}{\mathcal{M}}

\newcommand{\cF}{\mathcal{F}}

\newcommand\1{{\ensuremath {\mathds 1} }}
\newcommand{\Sph}{\mathbb{S}}

\newcommand\cS{\mathcal{S}}

\DeclareMathOperator{\im}{Im}

\DeclareMathOperator{\re}{Re}

\DeclareMathOperator{\supp}{supp}

\newcommand{\cA}{\mathcal{A}}

\newcommand{\dps}{\displaystyle}

\title{Extremizers for the Airy--Strichartz inequality}

\author{Rupert L. Frank}
\address[R. Frank]{Mathematisches Institut, Ludwig-Maximilans Universit\"at M\"unchen, Theresienstr. 39, 80333 M\"unchen, Germany, and Department of Mathematics, California Institute of Technology, Pasadena, CA 91125, USA}
\email{r.frank@lmu.de}

\author{Julien Sabin}
\address[J. Sabin]{Laboratoire de Math\'ematiques d'Orsay, Universit\'e Paris-Sud, CNRS, Universit\'e Paris-Saclay, 91405 Orsay, France}
\email{Julien.Sabin@math.u-psud.fr}

\begin{document}

\begin{abstract}
We identify the compactness threshold for optimizing sequences of the Airy--Strichartz inequality as an explicit multiple of the sharp constant in the Strichartz inequality. In particular, if the sharp constant in the Airy--Strichartz inequality is strictly smaller than this multiple of the sharp constant in the Strichartz inequality, then there is an optimizer for the former inequality. Our result is valid for the full range of Airy--Strichartz inequalities (except the endpoints) both in the diagonal and off-diagonal cases.
\end{abstract}

\maketitle

\section{Introduction}

The solution of the Cauchy problem for the Airy equation
\begin{equation}\label{eq:airy}
  \begin{cases}
   \partial_t v +\partial_x^3 v = 0,\quad t\in\R,\quad x\in\R, \\
   v_{|t=0}=u\in L^2_x(\R)
  \end{cases}
\end{equation}
may be written as $v(t,x)=(e^{-t\partial_x^3}u)(x)$. This solution disperses as $|t|\to\ii$, as reflected by the Airy--Strichartz inequalities
\begin{equation}\label{eq:A-S}
  \int_\R\left(\int_\R\left||D_x|^\gamma(e^{-t\partial_x^3}u)(x)\right|^q\,dx\right)^{p/q} dt\le C\norm{u}_{L^2_x}^p
\end{equation}
due to Kenig, Ponce, and Vega \cite{KenPonVeg-91}. This inequality holds for any exponents satisfying the relations
\begin{equation}\label{eq:constraint-exponents}
2\le p,q<\infty,\quad-\gamma+\frac3p+\frac1q=\frac12,\quad -\frac12<\gamma\le\frac1p \,;
\end{equation}
see \cite[Thm. 2.1]{KenPonVeg-91} for the case $\gamma=1/p$. The other cases of the inequality can be obtained from this case using Sobolev's embedding theorem as in \cite[Thm. 2.4]{KenPonVeg-91}.

We will be interested in the optimal constant in \eqref{eq:A-S}, that is,
\begin{equation}\label{eq:best-airy}
 \cA_p:=\sup_{u\neq0}\frac{\dps\int_\R\left(\int_\R\left||D_x|^{1/p}(e^{-t\partial_x^3}u)(x)\right|^qdx\right)^{p/q}dt}{\norm{u}_{L^2}^p},
\end{equation}
where the supremum is taken over all \emph{complex valued} functions $u$. In \eqref{eq:best-airy}, we considered the case $\gamma=1/p$, which is critical in a certain sense and has a richer behavior than the case $\gamma<1/p$. Therefore we mostly focus on this case and comment briefly on the necessary modifications for $\gamma<1/p$ in Appendix \ref{sec:subcritical}. Our goal is to investigate the existence of maximizers for the maximization problem $\cA_p$ and, more generally, to understand the behaviour of maximizing sequences.

For $\gamma=1/p$ the constraint \eqref{eq:constraint-exponents} on the exponents $p,q$ becomes $2/p+1/q=1/2$, which is the same condition under which Strichartz inequalities are valid for the Schr\"odinger equation in one space dimension
\begin{equation}\label{eq:schrodinger}
  \begin{cases}
   i\partial_tv-3\partial_x^2v=0,\quad t\in\R,\quad x\in\R,\\
   v_{|t=0}=u\in L^2_x(\R).
  \end{cases}
\end{equation}
For this last equation, the solution may be written as $v(t,x)=(e^{-3it\partial_x^2}u)(x)$, and the Strichartz inequalities \cite{Strichartz-77,GinVel-79a,KeeTao-98} correspond to the finiteness of
\begin{equation}\label{eq:best-schro}
 \cS_p:=\sup_{u\neq0}\frac{\dps\int_\R\left(\int_\R\left|(e^{-3it\partial_x^2}u)(x)\right|^qdx\right)^{p/q}\,dt}{\norm{u}_{L^2}^p},
\end{equation}
with, as we said, the constraint $2/p+1/q=1/2$. 

The Airy--Strichartz inequality \eqref{eq:A-S} belongs to the wider class of \emph{Fourier extension inequalities}. Indeed, the space-time Fourier transform of $e^{-t\partial_x^3}u$ is supported on the curve $\eta=\xi^3$. Hence, the estimate \eqref{eq:A-S} may be seen as an information about the decay of the Fourier transform of functions supported on this cubic curve. Such estimates have a long history in harmonic analysis. For instance, when the curve is replaced by a compact hypersurface with non-zero curvature, this decay is the content of the Stein--Tomas theorem \cite{Stein-86,Tomas-75}. When the hypersurface is the paraboloid $\eta=3\xi^2$, this corresponds to the problem \eqref{eq:best-schro} for the Schr\"odinger equation. 

Lately, there have been several works concerning extremizers for such Fourier extension inequalities. One group of works tries to find the sharp values of the constants and to identify all optimizers. This can be done in some specific cases; for instance, Foschi \cite{Foschi-07} (see also \cite{HunZha-06,BenBezCarHun-09,Goncalves-17}) proved that for $p=q=6$ the only optimizers for \eqref{eq:best-schro} up to symmetries are Gaussians, which is also the case for $p=2q=4$ as proved in \cite{BenBezCarHun-09,Carneiro-09}. A similar result holds for the two-dimensional Strichartz inequality (with $p=q=4$) \cite{Foschi-07} and for the two-dimensional Stein--Tomas inequality on the sphere \cite{Foschi-15}, with constants being the only optimizers up to symmetries. The corresponding questions in other dimensions are still open, with some recent progress for the one-dimensional Stein--Tomas inequality in \cite{CarFosOliThi-17}. All these results rely on the evenness of the exponent of the Lebesgue space in which the Fourier extension lives, in order to apply the Parseval identity followed by some convolution identities. 

A second group of works examines the behaviour of extremizing sequences from the point of view of concentration-compactness. This originated in work by Kunze \cite{Kunze-03} for the problem \eqref{eq:best-schro} in one space dimension with $p=q=6$ and by Shao \cite{Shao-09} in higher dimensions (see also \cite{Quilodran-13} when the paraboloid is replaced by the cone, and \cite{JiaPauSha-10,JiaShaSto-14} for a fourth order version). For the Stein--Tomas problem on the sphere, this was carried out by Christ and Shao \cite{ChrSha-12a} in dimension two and by Shao \cite{Shao-15} in dimension one. In these last two works, the evenness of the Lebesgue exponent again plays a crucial role. The case of higher dimensions (without evenness of the Lebesgue exponent) was treated in \cite{FraLieSab-16}, which is the work that we rely on. The weakness of these methods is that they merely give the existence of optimizers, without identifying them. On the other hand, they give universal informations about optimizing sequences. Studying optimizing sequences may also lead to the non-existence of optimizers \cite{Quilodran-15,CarOli-16,CarOliSou-17}.

The question of existence of optimizers for the Airy--Strichartz inequality \eqref{eq:A-S} in the case $p=q=6$, $\gamma=1/p=1/6$ has already been considered by Shao \cite{Shao-09b}. However, we believe that his argument is not complete. Therefore, our goal in this paper is two-fold: first, to correct and complete the argument in \cite{Shao-09b} and, second, to consider arbitrary exponents $p>4$.

In general, the strategy to prove the existence of optimizers is to obtain some kind of compactness for optimizing sequences. There may be several ways to lose this compactness, and the heart of the matter is to identify them exactly. Sometimes, this loss of compactness comes from exact symmetries of the inequality and cannot be avoided. In this case, the best result one may hope for is precompactness of optimizing sequences up to applying well-chosen symmetry transformations to the sequence. What makes the problem interesting and difficult is the presence of `approximate symmetries' which cannot be removed by applying symmetry transformations and which might lead to a more subtle loss of compactness. We will discuss those in more detail below.

In our case, the Airy--Strichartz inequality \eqref{eq:A-S} has a large group of (exact) symmetries: for any $(t_0,x_0,\lambda_0)\in\R\times\R\times(0,\ii)$ and any $u\in L^2(\R)$, define the transformation
$$\forall x\in\R,\quad [g_{t_0,x_0,\lambda_0}u](x)=\lambda_0^{1/2}(e^{-t_0\partial_x^3}u)(\lambda_0x+x_0)$$
and the associated group $G$ of all these transformations
$$G:= \left\{g_{t,x,\lambda},\, (t,x,\lambda)\in\R\times\R\times(0,+\ii)\right\}.$$
The inequality \eqref{eq:A-S} is invariant under the group $G$, in the sense that for any $u\in L^2(\R)$ and for any $g\in G$, we have
$$\norm{\Psi_p[gu]}_{L^p_tL^q_x}=\norm{\Psi_p[u]}_{L^p_tL^q_x},\quad \norm{gu}_{L^2_x}=\norm{u}_{L^2_x},$$
where we used the shortcut notation for $(t,x)\in\R\times\R$,
$$\Psi_p[u](t,x):=|D_x|^{1/p}e^{-t\partial_x^3}u(x)=\frac{1}{\sqrt{2\pi}}\int_\R|\xi|^{1/p}e^{it\xi^3+ix\xi}\hat{u}(\xi)\,d\xi,$$
where $\hat{u}$ denotes the Fourier transform of $u$,
$$\forall \xi\in\R,\quad \hat{u}(\xi):=\frac{1}{\sqrt{2\pi}}\int_\R u(x)e^{-ix\xi}\,dx.$$
Motivated by the previous discussion, we introduce the following notion of compactness.

\begin{definition}
 A sequence $(u_n)\subset L^2(\R)$ is \emph{precompact up to symmetries} if there are $(g_n)\subset G$ such that $(g_nu_n)$ is precompact in $L^2(\R)$. 
\end{definition}

\begin{definition}
 A \emph{maximizing sequence} for $\cA_p$ is a sequence $(u_n)\subset L^2_x(\R)$ with $\norm{u_n}_{L^2_x}=1$ such that $\norm{\Psi_p[u_n]}_{L^p_tL^q_x}^p\to\cA_p$ as $n\to\ii$.
\end{definition}

With these notions at hand, we may state our main result.

\begin{theorem}\label{thm:main}
 Let $4<p<\infty$ and $q$ such that $2/p+1/q=1/2$. Then, all maximizing sequences for $\cA_p$ are precompact up to symmetries if and only if
 \begin{equation}\label{eq:main-condition}
  \cA_p>a_p\cS_p,
 \end{equation}
 where $\cA_p$ is defined in \eqref{eq:best-airy}, $\cS_p$ is defined in \eqref{eq:best-schro}, and
  \begin{equation}\label{eq:def-aq}
  a_p:= \frac{2^{p/2}}{\pi^{p/(2q)}} \left( \frac{\Gamma(\frac{q+1}{2})}{\Gamma(\frac{q+2}{2})} \right)^{p/q} =  
  \left(\frac{1}{2\pi}\int_0^{2\pi}(1+\cos\theta)^{q/2}\,d\theta\right)^{p/q}>1.
 \end{equation}
In particular, if \eqref{eq:main-condition} holds, then there is a maximizer of $\cA_p$. 
\end{theorem}

\begin{remark}
 Since the evolution \eqref{eq:airy} preserves real-valuedness of functions, it is also natural to define the optimal constant
\begin{equation}\label{eq:best-airy-real}
 \cA_{p,\R}:=\sup_{u\in L^2_x(\R,\R)\setminus\{0\}}\frac{\dps\int_\R\left(\int_\R\left||D_x|^{1/p}(e^{-t\partial_x^3}u)(x)\right|^qdx\right)^{p/q}\,dt}{\norm{u}_{L^2}^p} \,,
\end{equation}
where the supremum is taken only over all \emph{real valued} functions $u$. Then, the exact same statement holds, replacing $\cA_p$ by $\cA_{p,\R}$, as we will show in Lemma \ref{lem:real-complex} at the end of Section \ref{sec:MMM}. (For the precompactness statement, we note also that the symmetry group $G$ preserves real-valuedness of functions.) Actually, we even have
$$
\cA_p=\cA_{p,\R}
$$
and there is a maximizer for $\cA_{p,\R}$ if and only if there is one for $\cA_p$. This extends to general extension problems, as we discuss in Appendix \ref{sec:symm}. The question whether $\cA_p$  is equal to $\cA_{p,\R}$ was raised by a referee, to whom we are most grateful, after an earlier version of this manuscript was submitted. We solved this problem using a technique from our previous paper \cite{FraLieSab-16}, but then learned by personal communication from D. Oliveira e Silva and R. Quilodr\'an, to whom we are also grateful, that the same proof was found independently in \cite{BroOliQui-18}.
\end{remark}

\begin{remark}
 In \cite{Shao-09b}, the necessary and sufficient condition for the existence of maximizers for $p=q=6$ are claimed to be $\cA_6>\cS_6$ and $\cA_{6,\R}>2^{-1/2}\cS_6$ and hence are not correct since $a_6=5/2$. We comment on this difference in Remark \ref{rk:shao}.
\end{remark}

\begin{remark}
In Theorem \ref{thm:subcrit} in the appendix we prove an analogous result for \eqref{eq:A-S} with $\gamma<1/p$. In this case maximizing sequences are always precompact up to symmetries, without an assumption like \eqref{eq:main-condition}.
\end{remark}

\begin{remark}
Since the constant $a_p$ is explicit, the inequality \eqref{eq:main-condition} may be tested once knowing an upper bound on $\cS_p$ and a lower bound for $\cA_p$. For $p=6$ and $p=8$ the maximization problem for $\cS_p$ is solved by Gaussians, which leads to an explicit value for the number $\cS_p$. It sounds natural (but is not proved so far) that $\cS_p$ is maximized by Gaussians for all $p>4$. This would turn the right side of \eqref{eq:main-condition} into an explicit number and then the condition can be tested by using trial function for the $\cA_p$ problem. To our knowledge, there is no conjecture on the precise form of maximizers for $\cA_p$.
\end{remark}

This theorem is the analogue of \cite[Thm 1.1]{FraLieSab-16}, where compactness of maximizing sequences is also conditional on a strict comparison inequality between the full problem and the problem on the paraboloid. This is due to the fact that both the cubic curve $\eta=\xi^3$ and the sphere are curved, implying that the paraboloid is the `local' model of these hypersurfaces (except at $\xi=0$ for the cubic curve). A maximizing sequence which would concentrate at a point would then, after rescaling, only see the problem for the paraboloid (see Remark \ref{rk:concentration-one-point} below) and hence be a test function for $\cS_p$. However, concentration around a point is clearly an obstacle to compactness, which explains the presence of $\cS_p$ in the condition \eqref{eq:main-condition} to rule out this concentration around a point. From a broader perspective, relating compactness to a strict `energy' inequality is a standard fact in several optimization problems: to prove the existence of a ground state for Schr\"odinger operators \cite[Thm. 11.5]{LieLos-01}, in the Br\'ezis-Nirenberg \cite[Lem. 1.2]{BreNir-83} or Yamabe \cite{Aubin-76} problems, and is even one of the main building blocks of the concentration-compactness theory of Lions \cite{Lions-84} (the so-called `strict sub-additivity conditions'). Let us notice that the idea of `embedding' the Schr\"odinger equation in the Airy equation also appeared in a nonlinear setting, as was noted for instance in \cite{Tao-07,ChrColTao-03,KilShaVis-12}.

What is peculiar in the strict inequality \eqref{eq:main-condition} is the presence of the constant $a_p$. The fact that $a_p$ is strictly larger than $1$ follows from the strict convexity of $x\mapsto x^{q/2}$ and Jensen's inequality since $d\theta/(2\pi)$ is a probability measure on $[0,2\pi]$.
One of the striking features of our problem is that the main enemy is not concentration around one point, but rather \emph{concentration around two opposite points} (or antipodal in the case of the sphere). As it turns out (see Lemma \ref{lem:energy-two-bubbles} and Remark \ref{rk:concentration-one-point} below), it is `energetically' more favorable to concentrate around such a pair a points (with an equally split $L^2$-mass), rather than at a single point. This is related to the fundamentally non-local nature of the Fourier transform, which implies that `bubbles' concentrating at different points on the hypersurface may have a Fourier transforms that `interact', and in the case of these two special points, that may interact `attractively'. This mechanism has been discovered in \cite{ChrSha-12a} in the case of the sphere, and does not occur for the paraboloid, since no pair of points on the paraboloid have opposite normals. It does not happen either for other model optimization problems like sharp Sobolev (or Br\'ezis-Nirenberg, Yamabe inequalities), since there the underlying operator is `local' (the embedding $H^1$ or $\dot{H}^1\hookrightarrow L^q$). Such two-point resonance requires specific tools, that we started to develop in \cite{FraLieSab-16}, and investigate further in the present article.

In a certain sense our present work is complementary to \cite{FraLieSab-16}. The loss of compactness that we have to overcome here comes from the fact that the modulation symmetry of the Strichartz inequality \eqref{eq:best-schro} is broken for the Airy equation and becomes an `approximate symmetry' for \eqref{eq:best-airy}. The scaling symmetry of \eqref{eq:best-schro}, however, is still an exact symmetry for \eqref{eq:best-airy}. Conversely, in our work on the Stein--Tomas inequality the scaling symmetry was an `approximate symmetry', while the modulation symmetry was an exact symmetry. For this reason the problems are different on a technical level.

It is worth mentioning that the condition \eqref{eq:main-condition} does not appear in the works \cite{ChrSha-12a} or \cite{Shao-15}: there, precompactness of optimizing sequences is unconditional. Likewise, we were able in \cite{FraLieSab-16} to prove that the condition analogue to \eqref{eq:main-condition} on the sphere was satisfied if maximizers of the Strichartz inequality \eqref{eq:best-schro} were Gaussians, as is known in dimension one and two. Indeed, if $\cS_p$ is attained for Gaussians, then one may `glue' two Gaussians at two antipodal points on the sphere at a scale $\epsilon>0$. As $\epsilon\to0$, the `energy' of such a test function converges to $a_p\cS_p$ (as in Lemma \ref{lem:energy-two-bubbles}). One is even able to compute the next order in $\epsilon$ of the energy, which turns out to be positive: this implies the strict inequality \eqref{eq:main-condition} on the sphere. In our setting of the cubic curve $\eta=\xi^3$, one may try the same strategy since we know that in one space dimension and for $p=q=6$ the optimizers of \eqref{eq:best-schro} are Gaussians. Unfortunately, the next order of the energy is negative, which does not allow to deduce \eqref{eq:main-condition}.

A standard approach to proving results of the flavor of Theorem \ref{thm:main} is to consider a maximizing sequence and to `follow the mass'. We show that either the mass escapes to small scales around a fixed frequency (which is then ruled out by condition \eqref{eq:main-condition}), or it must be precompact up to symmetries. Such a strong dichotomy result for quite general sequences (being a maximizing sequence does not tell much about the sequence) usually follows from refined versions of the considered inequality. We prove such a refined Airy--Strichartz estimate in Section \ref{sec:refined}. Since we are in one space dimension, a refined inequality may be proved in a quite elementary fashion using the Hausdorff--Young inequality. In higher dimensions, it often relies on deep bilinear estimates such as the ones obtained by Tao \cite{Tao-03}.

Once concentration to small scales is ruled out, the refined inequality typically implies the existence of a non-zero \emph{weak} limit up to symmetry for the maximizing sequence. Upgrading this weak convergence to strong convergence is the content of the \emph{Method of the Missing Mass} that was invented by Lieb \cite{Lieb-83b} in the context of the Hardy--Littewood--Sobolev inequality and used later on, for instance, in \cite{FraLie-12,FraLie-15,FraLieSab-16}. A useful tool in order to apply this method is the Br\'ezis--Lieb lemma \cite{Lieb-83b,BreLie-83}. Here, we need a slightly more general version of this lemma (due to both the presence of mixed Lebesgue spaces and of these two attractively interacting bubbles), that we prove in Appendix \ref{app:BL}. 

The article is organized as follows. In Section \ref{sec:MMM}, we apply and adapt the Method of the Missing Mass to our context. In particular, we choose to present the full proof of Theorem \ref{thm:main} first, and to prove the needed technical results in the following sections. We provide the proof in the complex-valued case, and only make remarks on why it also works in the real-valued case. In Section \ref{sec:refined}, we present the refined Airy--Strichartz inequality. In Section \ref{sec:approximate-operators}, we study some approximate Airy--Strichartz maps and provide some convergence results about them. These results, although a bit technical, carry some crucial features of the proof. Finally, some auxiliary results are provided in the appendices. 

\medskip

\noindent\textbf{Acknowledgments.} R.L.F. would like to thank Terence Tao for suggesting to look at this problem for general $p$ and to Diogo Oliveira e Silva, R\'ene Quilodr\'an and an anonymous referee for discussions concerning the $\cA_{p,\R}$ problem. Partial support through US National Science Foundation grant DMS-1363432 (R.L.F.) is also acknowledged.


\section{Proof of Theorem \ref{thm:main}: Method of the Missing Mass}\label{sec:MMM}

The Method of the Missing Mass shows that a sufficient condition for precompactness up to symmetries is the existence of a non-zero weak limit up to symmetries. As a consequence, we will need the following definition.

\begin{definition}
 Let $(u_n)\subset L^2(\R)$. We write $u_n\rightharpoonup_{\text{sym}}0$ if for all $(g_n)\subset G$ we have $g_nu_n\rightharpoonup0$ in $L^2(\R)$.
\end{definition}

We also define the `highest energy of sequences that are not precompact up to symmetry':

$$\cA^*_p:=\sup\left\{\limsup_{n\to\ii}\int_\R\left(\int_\R|\Psi_p[u_n]|^q\,dx\right)^{p/q}\,dt \,:\,\quad\norm{u_n}_{L^2}=1,\,\quad u_n\rightharpoonup_{\text{sym}}0\right\}.$$

Our main result is to identify exactly $\cA_p^*$.

\begin{theorem}\label{thm:MMM-2}
 Let $p>4$ and $q$ such that $2/p+1/q=1/2$. Then, we have
 $$
 \cA^*_p=a_p\cS_p \,,
 $$
 where $\cS_p$ is defined in \eqref{eq:best-schro} and $a_p$ is defined in \eqref{eq:def-aq}. Furthermore, the supremum defining $\cA^*_p$ is attained, that is, there is a sequence $(u_n)\subset L^2(\R)$ with $\norm{u_n}_{L^2}=1$ for all $n$, such that  $u_n\rightharpoonup_{\text{sym}}0$, and $\limsup_{n\to\ii}\norm{\Psi_p[u_n]}_{L^p_tL^q_x}^p=\cA_p^*$. 
\end{theorem}

\begin{remark}
One can also define a quantity $\cA_{p,\R}^*$ analogously to $\cA_p^*$ by restricting to sequences of real-valued functions. Then the exaxt same statement as in Theorem \ref{thm:MMM-2} holds, replacing $\cA_p^*$ by $\cA_{p,\R}^*$. In fact, we clearly have $\cA_{p,\R}^*\ge \cA_p^*$ and, since the sequence from Lemma \ref{lem:energy-two-bubbles} is real-valued, the proof of Theorem \ref{thm:MMM-2} will actually show that $\cA_{p,\R}^*\le a_p \cS_p$. Therefore, the claim in the real-valued case follows from that in the complex-valued case.
\end{remark}

Before proving Theorem \ref{thm:MMM-2}, we explain how it implies Theorem \ref{thm:main} using the Method of the Missing Mass. 

\begin{proposition}\label{prop:MMM}
 Let $p>4$ and $q$ such that $2/p+1/q=1/2$. Then, all normalized maximizing sequences for $\cA_p$ are precompact up to symmetries if and only if
 $$\cA_p>\cA^*_p.$$
\end{proposition}

Clearly, precompactness up to symmetries of maximizing sequences for $\cA_p$ implies the existence of a maximizer for $\cA_p$.

\begin{proof}
 Since we clearly always have $\cA_p\ge\cA^*_p$, the `only if' part follows from the definition of $\cA_p^*$ and the fact that the supremum defining it is attained, as stated in Theorem \ref{thm:MMM-2}. Indeed, a sequence $(u_n)$ satisfying $\norm{u_n}_{L^2}=1$ and $u_n\rightharpoonup_{\text{sym}}0$ cannot be precompact up to symmetries. We thus prove the `if' part. Assume that $\cA_p>\cA_p^*$, and let us apply the Method of the Missing Mass.
 Let $(u_n)$ an optimizing sequence for $\cA_p$ with $\norm{u_n}_{L^2}=1$. Since $\cA_p>\cA^*_p$, we have $u_n\not\rightharpoonup_{\text{sym}}0$ and hence there exists $(g_n)\subset G$ such that $g_nu_n\rightharpoonup v\neq0$ in $L^2(\R)$, perhaps up to a subsequence (the sequence $(g_nu_n)$ is bounded in $L^2(\R)$ and thus admits weak limits up to a subsequence). Let us write $v_n:=g_nu_n$ and  $r_n:=v_n-v\rightharpoonup0$ in $L^2(\R)$. We have 
 \begin{equation}\label{eq:mmm0}
    1=\norm{u_n}_{L^2}^2=\norm{v_n}_{L^2}^2=\norm{v}_{L^2}^2+\norm{r_n}_{L^2}^2+o_{n\to\ii}(1).
 \end{equation}
 Since $r_n\rightharpoonup0$ in $L^2(\R)$, we deduce that $\Psi_p[r_n]\to0$ a.e. in $\R^2$ by Lemma \ref{lem:compactness-smoothing}. Hence, by a variant of the Br\'ezis--Lieb lemma \cite{Lieb-83b,BreLie-83} for mixed Lebesgue spaces that we prove in Proposition \ref{bl}, we have for $\alpha=\min(p,q)$
 \begin{align*}
  \cA_p^{\alpha/p} &= \norm{\Psi_p[u_n]}_{L^p_tL^q_x}^\alpha+o_{n\to\ii}(1)\\
  &= \norm{\Psi_p[v_n]}_{L^p_tL^q_x}^\alpha+o_{n\to\ii}(1)\\
  &\le \norm{\Psi_p[v]}_{L^p_tL^q_x}^\alpha+\norm{\Psi_p[r_n]}_{L^p_tL^q_x}^\alpha+o_{n\to\ii}(1)\\
  &\le \cA_p^{\alpha/p}\left(\norm{v}_{L^2}^\alpha+\norm{r_n}_{L^2}^\alpha\right)+o_{n\to\ii}(1).
 \end{align*}
Using \eqref{eq:mmm0} and passing to the limit $n\to\ii$, we find, because of $\cA_p>0$,
$$
1\le\norm{v}_{L^2}^\alpha+(1-\norm{v}_{L^2}^2)^{\alpha/2}.
$$
 Since $p,q>2$, we always have $\alpha=\min(p,q)>2$ and thus $a^{\alpha/2}+b^{\alpha/2}\le(a+b)^{\alpha/2}$ for all $a,b\ge0$ with equality if and only if $a=0$ or $b=0$. Hence, we find that either $\norm{v}_{L^2}^2=0$ or $1-\norm{v}_{L^2}^2=0$. Since $v\neq0$, this means that $0=1-\norm{v}_{L^2}^2=\lim_{n\to\ii}\norm{r_n}_{L^2}^2=\lim_{n\to\ii}\norm{v_n-v}_{L^2}^2$, implying that $(v_n)=(g_nu_n)$ converges strongly in $L^2(\R)$. 
\end{proof}

Theorem \ref{thm:main} follows clearly from the combination of Proposition \ref{prop:MMM} and Theorem \ref{thm:MMM-2}. It remains to prove Theorem \ref{thm:MMM-2}, that is, to relate $\cA^*_p$ to the best constant in the Strichartz inequality. To make the constant $\cA^*_p$ more explicit, we prove that an optimizing sequence for $\cA^*_p$ must concentrate at two (distinct) opposite frequencies, which is reminiscent of what happens for the Stein--Tomas inequality \cite{ChrSha-12a,Shao-15,FraLieSab-16}. We first compute the `energy' of such a sequence. 

\begin{lemma}\label{lem:energy-two-bubbles}
 Let $p>4$ and $q$ such that $2/p+1/q=1/2$. Let $\chi\in L^2(\R)$, $\chi\neq0$, and $\epsilon>0$. Define
 $$\forall\xi\in\R,\quad\hat{\chi_\epsilon}(\xi)=\hat{\chi}\left(\frac{\xi-1}{\epsilon}\right)+\bar{\hat{\chi}\left(-\frac{\xi+1}{\epsilon}\right)}.$$
 Then, we have
 $$
 \frac{\dps\int_\R\left(\int_\R|\Psi_p[\chi_\epsilon](t,x)|^qdx\right)^{p/q}\,dt}{\norm{\chi_\epsilon}_{L^2}^p}\xrightarrow[\epsilon\to0]{}a_p\frac{\dps\int_\R\left(\int_\R|(e^{-3it\partial_x^2}\chi)(x)|^qdx\right)^{p/q}\,dt}{\norm{\chi}_{L^2}^p},
$$
where $a_p$ is defined in \eqref{eq:def-aq}.
\end{lemma}

\begin{remark}\label{rk:concentration-one-point}
 If the sequence concentrates at one frequency rather than two, by taking for instance $\hat{\chi_\epsilon}(\xi)=\hat{\chi}((\xi-1)/\epsilon)$, one has
 $$\frac{\dps\int_\R\left(\int_\R|\Psi_p[\chi_\epsilon](t,x)|^qdx\right)^{p/q}\,dt}{\norm{\chi_\epsilon}_{L^2}^p}\xrightarrow[\epsilon\to0]{}\frac{\dps\int_\R\left(\int_\R|(e^{-3it\partial_x^2}\chi)(x)|^qdx\right)^{p/q}\,dt}{\norm{\chi}_{L^2}^p}.$$
 In particular, since $a_p>1$ it is `energetically' more favorable to concentrate at two frequencies rather than one. Similarly, our analysis will show that concentrating at more than two points is `energetically' unfavorable.
\end{remark}

\begin{remark}
 It is important to notice that the transformation $\hat{\chi}\to\hat{\chi}((\cdot-1)/\epsilon)$ is not a symmetry of the Airy--Strichartz inequality, which means that translations in frequency (Galilean boosts) are not symmetries of the Airy--Strichartz inequality. By taking the limit of such a translation to infinity, one finds an effective `energy', the Strichartz energy for the standard Schr\"odinger evolution, which is now invariant by Galilean boosts. They are thus `approximate symmetries' of the Airy--Strichartz inequality.
\end{remark}

\begin{remark}\label{rk:shao}
A computation similar to the one in Lemma \ref{lem:energy-two-bubbles} appears in \cite[Lem. 6.1]{Shao-09b}, but with a different result. We believe that the problem in \cite[Lem. 6.1]{Shao-09b} is the passage from Eq. (89) to Eq. (90). For the same reason there is a problem in \cite[Lem. 5.1]{Shao-09b}, because the $L^6_{t,x}$-norm does not split, for instance, when the sequence concentrates at two opposite frequencies. This does not affect the validity of \cite[Thm. 1.5]{Shao-09b}, but it does affect the validity of \cite[Thm. 1.9]{Shao-09b}, since \cite[Lem. 5.1]{Shao-09b} is used in its proof. Let us also notice that in the real-valued version of the profile decomposition of \cite{Shao-09b}, two opposite frequencies necessarily carry the same profile. 
\end{remark}

\begin{proof}[Proof of Lemma \ref{lem:energy-two-bubbles}]
 First, we clearly have 
 $$\norm{\chi_\epsilon}_{L^2}^2=2\epsilon\norm{\chi}_{L^2}^2+o_{\epsilon\to0}(1).$$
 Next, we have
 \begin{align*}
  \Psi_p[\chi_\epsilon](t,x) &= 2\text{Re}\,\frac{\epsilon}{\sqrt{2\pi}}\int_\R|1+\epsilon\xi|^{1/q}e^{it(1+\epsilon\xi)^3+ix(1+\epsilon\xi)}\hat{\chi}(\xi)\,d\xi\\
  &= 2\text{Re}\,\frac{\epsilon}{\sqrt{2\pi}}e^{i(x+t)}\int_\R|1+\epsilon\xi|^{1/q}e^{it(3\epsilon^2\xi^2+\epsilon^3\xi^3)+i\epsilon\xi(x+3t)}\hat{\chi}(\xi)\,d\xi.
 \end{align*}
Changing variables $x\to x-3t$ and then $(x,t)\to(x/\epsilon,t/\epsilon^2)$, we find
\begin{align*}
  \norm{\Psi_p[\chi_\epsilon]}_{L^p_tL^q_x} &= \epsilon^{\frac12}\norm{2\text{Re}\,\frac{1}{\sqrt{2\pi}}e^{i\left(\frac{x}{\epsilon}-\frac{2t}{\epsilon^2}\right)}\int_\R|1+\epsilon\xi|^{1/q}e^{it(3\xi^2+\epsilon\xi^3)+ix\xi}\hat{\chi}(\xi)d\xi}_{L^p_tL^q_x}\\
  &=\epsilon^{\frac12}\norm{2\text{Re}\,e^{i\left(\frac{x}{\epsilon}-\frac{2t}{\epsilon^2}\right)}(T_{q,\epsilon}\chi)(t,x)}_{L^p_tL^q_x},
\end{align*}
where the operator $T_{q,\epsilon}$ is defined in Section \ref{sec:approximate-operators}. In particular, using Lemma \ref{lem:approximate-operator}, we deduce that
$$\norm{\Psi_p[\chi_\epsilon]}_{L^p_tL^q_x}=\epsilon^{\frac12}\norm{2\text{Re}\,e^{i\left(\frac{x}{\epsilon}-\frac{2t}{\epsilon^2}\right)}(e^{-3it\partial_x^2}\chi)(x)}_{L^p_tL^q_x}+o_{\epsilon\to0}(\epsilon^{\frac 12}).$$
For a.e. $(t,x)\in\R^2$, the function $\theta\in(\R/(2\pi\Z))^2\mapsto|2\text{Re}\,e^{i\left(\theta_1-2\theta_2\right)}(e^{-3it\partial_x^2}\chi)(x)|^q$ is continuous, and its maximum, $2^q|(e^{-3it\partial_x^2}\chi)(x)|^q$, belongs to $L^{p/q}_tL^1_x(\R\times\R)$. Using Lemma \ref{lem:homogenization-mixed}, we deduce that
\begin{align*}
&  \lim_{\epsilon\to0}\int_\R\left(\int_\R|2\text{Re}\,e^{i\left(\frac{x}{\epsilon}-\frac{2t}{\epsilon^2}\right)}(e^{-3it\partial_x^2}\chi)(x)|^q\,dx\right)^{p/q}\,dt\\
  &\qquad =\frac{1}{2\pi}\int_0^{2\pi}\int_\R\left(\frac{1}{2\pi}\int_0^{2\pi}\int_\R|2\text{Re}\,e^{i\left(\theta_1-2\theta_2\right)}(e^{-3it\partial_x^2}\chi)(x)|^q dx\,d\theta_1\right)^{p/q}\,dt\,d\theta_2\\
   &\qquad=\int_\R\left(\frac{1}{2\pi}\int_0^{2\pi}\int_\R|2\text{Re}\,e^{i\theta}(e^{-3it\partial_x^2}\chi)(x)|^q dx\,d\theta\right)^{p/q}\,dt.
\end{align*}
As a consequence, we have
\begin{align*}
 \norm{\Psi_p[\chi_\epsilon]}_{L^p_tL^q_x}^p 
 &= \epsilon^{p/2}\int_\R\left(\frac{1}{2\pi}\int_0^{2\pi}\int_\R|2\text{Re}\,e^{i\theta}(e^{-3it\partial_x^2}\chi)(x)|^qdx\,d\theta\right)^{p/q}\,dt+o_{\epsilon\to0}(\epsilon^{p/2})\\
 &= (2\epsilon)^{p/2}\left(\frac{1}{2\pi}\int_0^{2\pi}(1+\cos\theta)^{q/2}\,d\theta\right)^{p/q}\int_\R\left(\int_\R|(e^{-3it\partial_x^2}\chi)(x)|^qdx\right)^{p/q}\,dt \\
 & \qquad +o_{\epsilon\to0}(\epsilon^{p/2}).
\end{align*}
In the second equality we used the fact that for any $z\in\C$ we have
 \begin{equation}\label{eq:cosine0}
    \frac{1}{2\pi}\int_0^{2\pi}|e^{i\theta}z+e^{-i\theta}\overline z|^q\,d\theta = 2^{q/2} |z|^q \int_0^{2\pi}(1+\cos\theta)^{q/2}\,\frac{d\theta}{2\pi} \,.
 \end{equation}
This implies the result.
\end{proof}

We now turn to the proof of Theorem \ref{thm:MMM-2}, which interestingly uses a more involved version of the Method of the Missing Mass.

\begin{proof}[Proof of Theorem \ref{thm:MMM-2}]
 Let us first show that $\cA^*_p\ge a_p\cS_p$. For a sequence $(\epsilon_n)\subset(0,\ii)$ converging to $0$ and $\chi\in L^2(\R)$ with $\chi\neq 0$ we define $\chi_{\epsilon_n}$ as in Lemma \ref{lem:energy-two-bubbles} and set $u_n:= \chi_{\epsilon_n}/\norm{\chi_{\epsilon_n}}_{L^2}$.
 
Let us show that $u_n\rightharpoonup_{\text{sym}}0$. Hence, let $(g_n)=(g_{t_n,x_n,\lambda_n})\subset G$, and let us show that $g_nu_n\rightharpoonup0$ in $L^2(\R)$, which is equivalent to showing that any subsequence of $(g_nu_n)$ has a sub-subsequence converging weakly to zero in $L^2(\R)$. We show it for each one of the two bubbles composing $u_n$, which amounts to showing that 
 $$\xi\mapsto(\lambda_n\epsilon_n)^{-1/2}\hat{\chi}\left(\frac{\xi-\lambda_n}{\lambda_n\epsilon_n}\right)e^{i\frac{x_n}{\lambda_n}\xi+i\frac{t_n}{\lambda_n^3}\xi^3}$$
 has a subsequence converging weakly to zero as $n\to\ii$. Up to a subsequence, we have $\lambda_n\epsilon_n\to c$ as $n\to\ii$, with $c=0$, $0<c<\ii$, or $c=\ii$. In the cases $c=0$ or $c=\ii$, it is clear that it converges weakly to zero. If $0<c<\ii$, we must have $\lambda_n\to\ii$ and we thus also have weak convergence to zero.
 
According to Lemma \ref{lem:energy-two-bubbles} we have
$$
 \lim_{n\to\infty} \int_\R\left(\int_\R|\Psi_p[u_n](t,x)|^qdx\right)^{p/q}\,dt \xrightarrow[n\to\infty]{}a_p\frac{\dps\int_\R\left(\int_\R|(e^{-3it\partial_x^2}\chi)(x)|^qdx\right)^{p/q}\,dt}{\norm{\chi}_{L^2}^p}
$$
and therefore
$$
\cA^*_p \geq a_p\frac{\dps\int_\R\left(\int_\R|(e^{-3it\partial_x^2}\chi)(x)|^qdx\right)^{p/q}\,dt}{\norm{\chi}_{L^2}^p} \,.
$$ 
By taking the supremum over $\chi\in L^2(\R)$, $\chi\neq 0$, we conclude that $\cA^*_p\ge a_p\cS_p$.
 
Moreover, from \cite{Shao-09} we know that there is a maximizer for the Strichartz problem $\cS_p$, and taking $\chi$ in the above argument to be this maximizer we obtain a sequence $(u_n)$ with $\|u_n\|_{L^2}=1$ and $u_n \rightharpoonup_{\text{sym}}0$ such that $\lim_{n\to\infty} \| \Psi_p[u_n] \|_{L^p_t L^q_x}^p \to a_p\cS_p$. Thus, once we have shown the reverse inequality $\cA^*_p\le a_p\cS_p$, we have also proved the last statement in Theorem~\ref{thm:MMM-2}.

Thus, it remains to show the inequality $\cA^*_p\le a_p\cS_p$. Let $(u_n)\subset L^2(\R)$ be a sequence such that $\norm{u_n}_{L^2}=1$ and $u_n\rightharpoonup_{\text{sym}}0$, satisfying
$$
\limsup_{n\to\ii}\int_\R\left(\int_\R|\Psi_p[u_n]|^qdx\right)^{p/q}dt\ge\frac12\cA_p^*.
$$
We decompose $u_n$ as
$$
u_n=u_{n,>}+u_{n,<}
$$
with $\hat{u_{n,>}}=\1_{\R_+}\hat{u_n}$ and $\hat{u_{n,<}}=\1_{\R_-}\hat{u_n}$. Since $\Psi_p[u_n]=\Psi_p[u_{n,>}]+\Psi_p[u_{n,<}]$, we have
$$
\left(\int_\R|\Psi_p[u_n]|^qdx\right)^{p/q} \le 2^{p-1} \left( \left(\int_\R|\Psi_p[u_{n,>}]|^qdx\right)^{p/q} + \left(\int_\R|\Psi_p[u_{n,<}]|^qdx\right)^{p/q} \right).
$$
Integrating this inequality with respect to $t$ and estimating the right side, we find
$$
\norm{\Psi_p[u_{n,>}]}_{L^p_tL^q_x}^p 
\le 2^p \max\left\{ \norm{\Psi_p[u_{n,>}]}_{L^p_tL^q_x}^p,\,\norm{\Psi_p[u_{n,>}]}_{L^p_tL^q_x}^p \right\} \,.
$$
Passing to a subsequence and replacing $\hat{u_n}$ to $\hat{u_n}(-\cdot)$ if necessary (which is still an admissible sequence for $\cA^*_p$), we may assume that the maximum is always attained at $\norm{\Psi_p[u_{n,>}]}_{L^p_tL^q_x}^p$. Thus,
$$
\limsup_{n\to\ii}\norm{\Psi_p[u_{n,>}]}_{L^p_tL^q_x}^p\ge\frac{1}{2^{p+1}}\cA^*_p
$$
and, in particular, $\Psi_p[u_{n,>}]\nrightarrow0$ in $L^p_tL^q_x$.

By Corollary \ref{coro:refined}, there are $(g_n)\subset G$ and $(\eta_n)\subset\R_+$ with $\eta_n\ge1/2$ such that the sequence $((\hat{g_nu_{n,>}})(\cdot+\eta_n))$ has a weak limit $\hat{v_>}\neq0$ in $L^2(\R)$, with a lower bound 
 $$\norm{v_>}_{L^2}\ge \gamma>0,$$
 where $\gamma$ only depends on $p$. We have $\eta_n\to\ii$, for otherwise $\eta_n\to c\in\R$ up to a subsequence, and then $(g_nu_{n,>})$ has a non-zero weak limit, which then implies that $(g_nu_n)$ has a non-zero weak limit point, which contradicts $u_n\rightharpoonup_\text{sym}0$.
 
 Since $\supp\hat{g_nu_{n,>}}(\cdot+\eta_n)\subset[-\eta_n,+\ii)$, we may write
 \begin{align*}
    \hat{g_nu_{n,>}}(\eta+\eta_n) &=: \hat{v_>}(\eta)+\hat{r_{n,>}}(\eta)\\
    &= \1_{\eta+\eta_n\ge0}\hat{v_>}(\eta)+\1_{\eta+\eta_n\ge0}\hat{r_{n,>}}(\eta)\\
    &= \hat{v_>}(\eta)-\1_{\eta+\eta_n<0}\hat{v_>}(\eta)+\1_{\eta+\eta_n\ge0}\hat{r_{n,>}}(\eta)\\
    &=:\hat{v_>}(\eta)+\hat{r_{n,>}^{(2)}}(\eta)+\hat{r_{n,>}^{(1)}}(\eta),
 \end{align*}
 with $\hat{r_{n,>}^{(1)}}\rightharpoonup0$ in $L^2(\R)$, $\supp\hat{r_{n,>}^{(1)}}\subset[-\eta_n,+\ii)$ and $\hat{r_{n,>}^{(2)}}\to0$ strongly in $L^2(\R)$. Furthermore, the sequence $(\hat{g_nu_{n,<}}(-\cdot-\eta_n))$ is bounded in $L^2(\R)$, and hence admits a weak limit $\hat{v_<}$ (which may be zero) in $L^2(\R)$, up to a subsequence. We split accordingly
 $$\hat{g_nu_{n,<}}(-\cdot-\eta_n)=\hat{v_<}+\hat{r_{n,<}^{(2)}}+\hat{r_{n,<}^{(1)}}$$
 with $\hat{r_{n,<}^{(1)}}\rightharpoonup0$ in $L^2(\R)$, $\supp\hat{r_{n,<}^{(1)}}\subset[-\eta_n,+\ii)$ and $\hat{r_{n,<}^{(2)}}\to0$ strongly in $L^2(\R)$. Defining $\delta_n:=1/\eta_n\to0$, we now have
 \begin{align*}
  \norm{\Psi_p[u_n]}_{L^p_tL^q_x} &= \norm{\Psi_p[g_nu_{n,>}]+\Psi_p[g_nu_{n,<}]}_{L^p_tL^q_x}\\
  &= \norm{e^{ix/\delta_n-2it/\delta_n^2}T_{p,\delta_n}(v_>+r_{n,>}^{(1)}+r_{n,>}^{(2)}) \!+\! e^{-ix/\delta_n+2it/\delta_n^2}\bar{T_{p,\delta_n}(v_<+r_{n,<}^{(1)}+r_{n,<}^{(2)})}}_{L^p_tL^q_x}
 \end{align*}
with an operator $T_{p,\delta}$ defined in Section \ref{sec:approximate-operators}. By Lemma \ref{lem:local-convergence}, we deduce that
 $$
 T_{p,\delta_n}r_{n,>}^{(1)} \to 0
 \qquad\text{and}\qquad 
 T_{p,\delta_n}r_{n,<}^{(1)} \to 0
 $$
 almost everywhere in $(t,x)\in\R\times\R$. By Lemma \ref{lem:approximate-operator}, we also deduce that 
 $$
 T_{p,\delta_n}(v_>+r_{n,>}^{(2)}) \to e^{-3it\partial_x^2}v_>
 \qquad\text{and}\qquad 
T_{p,\delta_n}(v_<+r_{n,<}^{(2)}) \to e^{-3it\partial_x^2}v_<
 $$
 strongly in $L^p_tL^q_x$. Using Proposition \ref{bl} applied with
 $$
 \begin{cases}
    \Pi_n:=e^{ix/\delta_n-2it/\delta_n^2}e^{-3it\partial_x^2}v_>+e^{-ix/\delta_n+2it/\delta_n^2}\bar{e^{-3it\partial_x^2}v_<},\\
    \rho_n:=e^{ix/\delta_n-2it/\delta_n^2}T_{\delta_n}r_{n,>}^{(1)}+e^{-ix/\delta_n+2it/\delta_n^2}\bar{T_{\delta_n}r_{n,<}^{(1)}},
 \end{cases}
 $$
 we deduce that 
 \begin{equation}\label{eq:decomp-MMM-concentration}
  \norm{\Psi_p[u_n]}_{L^p_tL^q_x}^\alpha \le A_{1,n}+A_{2,n}+o_{n\to\ii}(1)
 \end{equation}
 with $\alpha=\min(p,q)$ and
 $$\begin{cases}
    A_{1,n}:=\norm{e^{ix/\delta_n-2it/\delta_n^2}e^{-3it\partial_x^2}v_>+e^{-ix/\delta_n+2it/\delta_n^2}\bar{e^{-3it\partial_x^2}v_<}}_{L^p_tL^q_x}^\alpha,\\
    A_{2,n}=\norm{e^{ix/\delta_n-2it/\delta_n^2}T_{\delta_n}r_{n,>}^{(1)}+e^{-ix/\delta_n+2it/\delta_n^2}\bar{T_{\delta_n}r_{n,<}^{(1)}}}_{L^p_tL^q_x}^\alpha \,.
 \end{cases}$$
 
 Using Lemma \ref{lem:homogenization-mixed} in the same way as in the proof of Lemma \ref{lem:energy-two-bubbles}, we find that 
 $$
    \lim_{n\to\ii}A_{1,n}=\left(\int_\R\left(\frac{1}{2\pi}\int_0^{2\pi}\int_\R\left|e^{i\theta}(e^{-3it\partial_x^2}v_>)(x)+e^{-i\theta}\bar{(e^{-3it\partial_x^2}v_<)(x)}\right|^q\,dx\,d\theta\right)^{p/q}\,dt\right)^{\alpha/p}.
 $$
 For any $z_1,z_2\in\C$ we have
 \begin{equation}\label{eq:cosine}
    \frac{1}{2\pi}\int_0^{2\pi}|e^{i\theta}z_1+e^{-i\theta}z_2|^q\,d\theta = (|z_1|^2+|z_2|^2)^{q/2}\int_0^{2\pi}(1+a\cos\theta)^{q/2}\,\frac{d\theta}{2\pi}
 \end{equation}
  with $a=2|z_1||z_2|/(|z_1|^2+|z_2|^2)\in[0,1]$. (Note that this is the generalization of \eqref{eq:cosine0} to $z_1\neq\overline{z_2}$.) As in the proof of Lemma 6.1 in \cite{FraLieSab-16}, this last function can be shown to be maximal at $a=1$. As a consequence,
 $$\frac{1}{2\pi}\int_0^{2\pi}|e^{i\theta}z_1+e^{-i\theta}z_2|^q\,d\theta\le(|z_1|^2+|z_2|^2)^{q/2}\int_0^{2\pi}(1+\cos\theta)^{q/2}\,\frac{d\theta}{2\pi}=a_p^{q/p}(|z_1|^2+|z_2|^2)^{q/2},$$
 and thus, by the triangle inequality in $L_t^{p/2}L_x^{q/2}$ (noting that $p,q\geq 2$),
 \begin{align*}
    \lim_{n\to\ii}A_{1,n} &\le a_p^{\alpha/p}\left(\int_\R\left(\int_\R\left(|e^{-3it\partial_x^2}v_>|^2+|e^{-3it\partial_x^2}v_<|^2\right)^{q/2}\,dx\right)^{p/q}\,dt\right)^{\alpha/p}\\
    &\le a_p^{\alpha/p}\left(\norm{e^{-3it\partial_x^2}v_>}_{L^p_tL^q_x}^2+\norm{e^{-3it\partial_x^2}v_<}_{L^p_tL^q_x}^2\right)^{\alpha/2}\\
    &\le a_p^{\alpha/p}\cS_p^{\alpha/p}\left(\norm{v_>}_{L^2}^2+\norm{v_<}_{L^2}^2\right)^{\alpha/2}.
 \end{align*}
 
 Concerning the term $A_{2,n}$, reversing the change of variables in $(t,x,\eta)$ that we have done, we find that 
 $$A_{2,n}=\norm{\Psi_p[w_n]}_{L^p_tL^q_x}^\alpha,$$
 with 
 \begin{align*}
  \hat{w_n}(\eta) &:=\hat{r_{n,>}^{(1)}}(\eta-\eta_n)+\hat{r_{n,<}^{(1)}}(-\eta-\eta_n)\\
  &=\hat{g_nu_n}(\eta)-\hat{v_>}(\eta-\eta_n)-\hat{v_<}(-\eta-\eta_n)-\hat{r_{n,>}^{(2)}}(\eta-\eta_n)-\hat{r_{n,<}^{(2)}}(-\eta-\eta_n).
 \end{align*}
 
 Let us show that $w_n\rightharpoonup_\text{sym}0$. Since $u_n\rightharpoonup_\text{sym}0$, we also have $g_nu_n\rightharpoonup_\text{sym}0$. As we have seen in the beginning of the proof, translating a fixed profile to infinity in frequencies gives a sequences that vanishes weakly up to any symmetry. Hence, the terms involving $v_>$ and $v_<$ also $\rightharpoonup_\text{sym}0$. Finally, the terms involving $r_{n,>}^{(2)}$ and $r_{n,<}^{(2)}$ converge strongly to zero in $L^2$, and thus also $\rightharpoonup_\text{sym}0$. Hence, we have proved that $w_n\rightharpoonup_\text{sym}0$.
 
 Since 
 $$\supp\hat{r_{n,>}^{(1)}}(\cdot-\eta_n)\subset\R_+,\qquad\supp\hat{r_{n,<}^{(1)}}(-\cdot-\eta_n)\subset\R_-,$$
 we deduce that
 $$\norm{w_n}_{L^2}^2=\norm{r_{n,>}^{(1)}}_{L^2}^2+\norm{r_{n,<}^{(1)}}_{L^2}^2.$$
 By the definitions of $r_{n,>}^{(1)}$ and $r_{n,<}^{(1)}$ and their weak convergence to zero, we also have
 $$\norm{r_{n,>}^{(1)}}_{L^2}^2=\norm{u_{n,>}}_{L^2}^2-\norm{v_>}_{L^2}^2+o(1),\quad \norm{r_{n,<}^{(1)}}_{L^2}^2=\norm{u_{n,<}}_{L^2}^2-\norm{v_<}_{L^2}^2+o(1),$$
 which implies that
 $$\norm{w_n}_{L^2}^2=\norm{u_n}_{L^2}^2-\norm{v_>}_{L^2}^2-\norm{v_<}_{L^2}^2+o(1)=1-\norm{v_>}_{L^2}^2-\norm{v_<}_{L^2}^2+o(1).$$
 If $\norm{v_>}_{L^2}^2+\norm{v_<}_{L^2}^2<1$, then $w_n/\norm{w_n}_{L^2}\rightharpoonup_\text{sym}0$ and by definition of $\cA^*_p$ we find that 
 $$\limsup_{n\to\ii}\norm{\Psi_p[w_n]}_{L^p_tL^q_x}^p\le\cA^*_p(1-\norm{v_>}_{L^2}^2-\norm{v_<}_{L^2}^2)^{p/2}.$$
 If $\norm{v_>}_{L^2}^2+\norm{v_<}_{L^2}^2=1$, then $w_n\to0$ strongly in $L^2$, so that the same inequality holds by the Airy--Strichartz inequality.
 
 We now insert our asymptotic estimates on $A_{1,n}$ and $A_{2,n}$ into \eqref{eq:decomp-MMM-concentration} and take the limit $n\to\ii$. This yields the inequality
 $$\limsup_{n\to\ii}\norm{\Psi_p[u_n]}_{L^p_tL^q_x}^\alpha\le a_p^{\alpha/p}\cS_p^{\alpha/p}\left(\norm{v_>}_{L^2}^2+\norm{v_<}_{L^2}^2\right)^{\alpha/2}+(\cA^*_p)^{\alpha/p}(1-\norm{v_>}_{L^2}^2-\norm{v_<}_{L^2}^2)^{\alpha/2},$$
 which we may rewrite as 
 \begin{multline*}
  \left(1-(1-\norm{v_>}_{L^2}^2-\norm{v_<}_{L^2}^2)^{\alpha/2}\right)(\cA^*_p)^{\alpha/p}-a_p^{\alpha/p}\cS_p^{\alpha/p}\left(\norm{v_>}_{L^2}^2+\norm{v_<}_{L^2}^2\right)^{\alpha/2}\\
  \le(\cA^*_p)^{\alpha/p}-\limsup_{n\to\ii}\norm{\Psi_p[u_n]}_{L^p_tL^q_x}^\alpha
 \end{multline*}
 By the elementary estimate $1-(1-x)^{\alpha/2}\ge x^{\alpha/2}$ valid for all $x\in[0,1]$ since $\alpha>2$, this implies that
 \begin{align*}
  (\cA^*_p)^{\alpha/p}-\limsup_{n\to\ii}\norm{\Psi_p[u_n]}_{L^p_tL^q_x}^\alpha &\ge \left(\norm{v_>}_{L^2}^2+\norm{v_<}_{L^2}^2\right)^{\alpha/2}((\cA^*_p)^{\alpha/p}-a_p^{\alpha/p}\cS_p^{\alpha/p})\\
  &\ge \gamma^\alpha((\cA^*_p)^{\alpha/p}-a_p^{\alpha/p}\cS_p^{\alpha/p}).
 \end{align*}
Taking the supremum over all such sequences $(u_n)$, the left side vanishes, indeed showing that $\cA^*_p\le a_p\cS_p$.    
\end{proof}

We end this section by some remarks about the real-valued case.

\begin{lemma}\label{lem:real-complex}
 For any triplet $(p,q,\gamma)$ with $p,q\ge2$, we have
 \begin{equation}\label{eq:real-complex}
    \sup_{u\in L^2(\R,\C)\setminus\{0\}}\frac{\norm{|D_x|^\gamma e^{-t\partial_x^3}u}_{L^p_tL^q_x}}{\norm{u}_{L^2_x}}=\sup_{u\in L^2(\R,\R)\setminus\{0\}}\frac{\norm{|D_x|^\gamma e^{-t\partial_x^3}u}_{L^p_tL^q_x}}{\norm{u}_{L^2_x}}.
 \end{equation}
Moreover, there is a maximizer for the supremum on the left side if and only if there is one for the supremum on the right side. 
\end{lemma}

\begin{proof}
 While the inequality $\ge$ is clear, let us show the inequality $\le$. Hence, let $u\in L^2(\R,\C)$, $u\neq0$. Splitting $u$ as $u=\re u +i\im u$ and using that the operator $|D_x|^\gamma e^{-t\partial_x^3}$ preserves real-valuedness, we deduce that 
 $$||D_x|^\gamma e^{-t\partial_x^3}u|^2=(|D_x|^\gamma e^{-t\partial_x^3}\re u)^2+(|D_x|^\gamma e^{-t\partial_x^3}\im u)^2,$$
 so that 
 $$\norm{|D_x|^\gamma e^{-t\partial_x^3}u}_{L^p_tL^q_x}=\norm{(|D_x|^\gamma e^{-t\partial_x^3}\re u)^2+(|D_x|^\gamma e^{-t\partial_x^3}\im u)^2}_{L^{p/2}_tL^{q/2}_x}^{1/2}.$$
 Using the triangle inequality in $L^{p/2}_tL^{q/2}_x$, we deduce
 \begin{align*}
  \norm{|D_x|^\gamma e^{-t\partial_x^3}u}_{L^p_tL^q_x} &\le \left(\norm{|D_x|^\gamma e^{-t\partial_x^3}\re u}_{L^p_tL^q_x}^2+\norm{|D_x|^\gamma e^{-t\partial_x^3}\im u}_{L^p_tL^q_x}^2\right)^{1/2}\\
  &\le \sup_{v\in L^2(\R,\R)\setminus\{0\}}\frac{\norm{|D_x|^\gamma e^{-t\partial_x^3}v}_{L^p_tL^q_x}}{\norm{v}_{L^2_x}}\left(\norm{\re u}_{L^2_x}^2+\norm{\im u}_{L^2_x}^2\right)^{1/2}.
 \end{align*}
Finally, using that $|u|^2=(\re u)^2+(\im u)^2$ and hence $\left(\norm{\re u}_{L^2_x}^2+\norm{\im u}_{L^2_x}^2\right)^{1/2}=\norm{u}_{L^2_x}$, we deduce the inequality $\le$ in \eqref{eq:real-complex} by taking the supremum over all $u$.

The equality of the suprema in \eqref{eq:real-complex} implies that, if $u$ is a maximizer for the supremum on the right side, then it is also a maximizer for the one on the left side. Conversely, by tracking the equality cases in the previous proof, one sees that if $u$ is a maximizer for the supremum on the left, then $\re u$ or $\im u$ is a maximizer for the supremum on the right.
\end{proof}

\begin{remark}
It is natural to wonder whether maximizers for $\mathcal A_p$ are necessarily complex multiples of real-valued functions. We do not know how to deduce this using the above method of proof. In fact, we from the equality case in the triangle inequality in $L^{p/2}_tL^{q/2}_x$ we learn that either $\im u=0$ or there is a $\lambda\ge 0$ such that $(|D_x|^\gamma e^{-t\partial_x^3} \re u)^2 =\lambda (|D_x|^\gamma e^{-t\partial_x^3} \im u)^2$ almost everywhere.
\end{remark}

Lemma \ref{lem:real-complex} shows that $\cA_p=\cA_{p,\R}$, from which we deduce the `if' part of the analogue of Theorem \ref{thm:main} in the real-valued case (because any real-valued maximizing sequence for $\cA_{p,\R}$ is then also a complex-valued maximizing sequence for $\cA_p$). The `only if' part follows from the fact that the sequence built in Lemma \ref{lem:energy-two-bubbles} is real-valued.

The rest of the article is devoted to the proofs of the results used during this last section.


\section{A refined Airy--Strichartz inequality}\label{sec:refined}

\subsection{Refined inequality and its consequences}

We begin with the compactness result that allows us to extract a non-trivial profile from an optimizing sequence. It will follow from a refined version of the Airy--Strichartz inequality. 

\begin{theorem}\label{thm:refined-stri}
 There are $\theta\in(0,1)$ and $C>0$ such that for all $u\in L^2(\R)$ we have 
 \begin{equation}\label{eq:refined-stri}
    \norm{|D_x|^{1/6}e^{-t\partial_x^3}u}_{L^6_{t,x}(\R\times\R)}\le C\left(\sup_{I\in\cD}|c(I)|^{-\frac16}|I|^{-\frac12}\norm{|D_x|^{1/6}e^{-t\partial_x^3}u_I}_{L^\ii_{t,x}}\right)^\theta\norm{u}_{L^2}^{1-\theta},
 \end{equation}
 where $\cD$ denotes the family of dyadic intervals of $\R$ (see Definition \ref{def:dyadic-interval} below), $|I|$ denotes the length of the interval $I$, $c(I)$ its center, and $\hat{u_I}:=\hat{u}_{|I}$.
\end{theorem}

Similar refined inequalities appeared previously in the literature, for instance in the context of the Sobolev inequality \cite{GerMeyOru-97} (see also \cite[Prop. 4.8]{KilVis-book}), the Stein--Tomas inequality \cite{MoyVarVeg-99,Oliveira-14,Shao-15,FraLieSab-16}, or the Strichartz inequality \cite{Bourgain-98,MerVeg-98,CarKer-07,BegVar-07}. In the context of the Airy equation, related refined inequalities appeared in \cite{KenPonVeg-00} (see \cite{Shao-09b} for a different proof). The particularity of our estimate \eqref{eq:refined-stri} is the presence of an $L^\ii_{t,x}$-norm on the right side, which provides a rather direct route to compactness as we will shortly see. The strategy via such $L^\ii_{t,x}$-norms has been initiated in \cite{Tao-09,KilVis-book}, and we followed it in \cite{FraLieSab-16} (see also, for instance, \cite{FraLie-12}). Contrary to the aforementioned works, the center $c(I)$ of the interval $I$ appears on the right side of \eqref{eq:refined-stri}; this is due to the fact that the Airy--Strichartz inequality is not invariant by translations in Fourier space. This is the first time we encounter such a phenomenon.

From the refined estimate in $L^6_{t,x}$, we deduce a refined estimate in mixed Lebesgue spaces $L^p_tL^q_x$ by complex interpolation. 

\begin{corollary}\label{coro:refined-pq}
  Let $p>4$ and $q$ such that $2/p+1/q=1/2$. Then, there exist $\theta\in(0,1)$ and $C>0$ such that for all $u\in L^2(\R)$ we have 
 \begin{equation}\label{eq:refined-stri-mixed}
    \norm{|D_x|^{1/p}e^{-t\partial_x^3}u}_{L^p_tL^q_x(\R\times\R)}\le C\left(\sup_{I\in\cD}|c(I)|^{-\frac16}|I|^{-\frac12}\norm{|D_x|^{1/6}e^{-t\partial_x^3}u_I}_{L^\ii_{t,x}}\right)^\theta\norm{u}_{L^2}^{1-\theta}.
 \end{equation}
\end{corollary}

\begin{proof}
When $p=6$, this is Theorem \ref{thm:refined-stri}. When $p\neq6$, we distinguish two cases. If $p>6$, we pick $\tilde{p}>p$ and $\tilde{q}$ such that $2/\tilde{p}+1/\tilde{q}=1/2$ so that by Proposition \ref{prop:complex-interpolation} we have
$$\norm{|D_x|^{1/p}e^{-t\partial_x^3}u}_{L^p_tL^q_x}\le\norm{|D_x|^{1/\tilde{p}}e^{-t\partial_x^3}u}_{L^{\tilde{p}}_tL^{\tilde{q}}_x}^\theta\norm{|D_x|^{1/6}e^{-t\partial_x^3}u}_{L^6_{t,x}}^{1-\theta},$$
where $\theta\in(0,1)$ is such that $1/p=\theta/\tilde{p}+(1-\theta)/6$. By the Airy--Strichartz inequality, we have 
$$\norm{|D_x|^{1/\tilde{p}}e^{-t\partial_x^3}u}_{L^{\tilde{p}}_tL^{\tilde{q}}_x}\le C\norm{u}_{L^2_x}.$$
The $L^6_{t,x}$-part may be estimated by Theorem \ref{thm:refined-stri}, which gives the result. When $p<6$, we interpolate in the same fashion $L^p_tL^q_x$ between $L^6_{t,x}$ and $L^{\tilde{p}}_t L^{\tilde{q}}_x$ with $4<\tilde{p}<p$, which gives the same result.
\end{proof}

Our main interest in proving the refined estimate \eqref{eq:refined-stri-mixed} is to find a non-trivial weak limit (a first 'profile' in the profile decomposition), in the case of non-vanishing sequences (in the language of concentration-compactness).

\begin{corollary}\label{coro:refined}
 Let $p>4$ and $q$ such that $2/p+1/q=1/2$. Let $(u_n)$ be a bounded sequence in $L^2(\R)$ with $\text{supp}\,\hat{u_n}\subset\R_\pm$ such that 
$$
|D_x|^{1/p}e^{-t\partial_x^3}u_n\nrightarrow0
\qquad\text{in}\ L^p_tL^q_x \,. 
 $$
Then, there are $(t_n,x_n,\xi_n,\lambda_n)\subset\R\times\R\times\R_{\pm}\times\R_+^*$ with $\lambda_n/|\xi_n|\le2$ such that, after passing to a subsequence if necessary, the sequence $((\hat{g_nu_n})(\cdot+\xi_n/\lambda_n))$ with $g_n=g_{t_n,x_n,\lambda_n^{-1}}$
 has a non-zero weak limit $v$ in $L^2(\R)$. Furthermore, if $a,b>0$ are such that 
 $$\limsup_{n\to\ii}\norm{|D_x|^{1/p}e^{-t\partial_x^3}u_n}_{L^p_tL^q_x}\ge a,\qquad \norm{u_n}_{L^2}\le b,$$
 then we have the estimate 
 $$\norm{v}_{L^2}\ge ca^\alpha b^{-\beta},$$
 where $c,\alpha,\beta>0$ depend only on $p$.
\end{corollary}

\begin{remark}
 In the previous statement, when we write $(\xi_n)\subset\R_\pm$, we mean that we may choose $(\xi_n)\subset\R_+$ if $\text{supp}\,\hat{u_n}\subset\R_+$ or $(\xi_n)\subset\R_-$ if $\text{supp}\,\hat{u_n}\subset\R_-$.
\end{remark}

\begin{proof}
By the refined estimate \eqref{eq:refined-stri-mixed}, there are $(t_n,x_n)\subset\R\times\R$ and dyadic intervals $I_n\subset\R_\pm$ such that, along a subsequence,
 $$|c(I_n)|^{-\frac16}|I_n|^{-\frac12}\left|\int_{I_n} |\xi|^{1/6}e^{ix_n\xi+it_n\xi^3}\hat{u_n}(\xi)\,d\xi\right|\ge\epsilon:=\left(\frac{a}{2Cb^{1-\theta}}\right)^{\frac1\theta} \,,$$
where $C,\theta$ are the constants appearing in \eqref{eq:refined-stri-mixed}. Denote by $\xi_n:=c(I_n)$ and $\lambda_n=|I_n|$ so that $I_n=\xi_n+\lambda_n(-1/2,1/2)$, and $0<\lambda_n/|\xi_n|\le2$ (see Definition \ref{def:dyadic-interval} below). Up to a subsequence, we may assume that $\lambda_n/|\xi_n|\to\delta\in[0,2]$. In the previous integral, write any $\xi$ as $\xi=\xi_n+\lambda_n\eta$ to obtain
 $$\left|\int_{(-1/2,1/2)}\left|1+\frac{\lambda_n}{\xi_n}\eta\right|^{1/6}e^{ix_n(\xi_n+\lambda_n\eta)+it_n(\xi_n+\lambda_n\eta)^3}\lambda_n^{1/2}\hat{u_n}(\xi_n+\lambda_n\eta)\,d\eta\right|\ge\epsilon$$
 for all $n$. The sequence
 $$v_n(\eta):=e^{ix_n(\xi_n+\lambda_n\eta)+it_n(\xi_n+\lambda_n\eta)^3}\lambda_n^{1/2}\hat{u_n}(\xi_n+\lambda_n\eta)$$
 is bounded in $L^2(\R)$. Pick any weak limit $v$ of it. Since we have the convergence
 $$\1_{(-1/2,1/2)}|1+(\lambda_n/|\xi_n|)\cdot|^{1/6}\to\1_{(-1/2,1/2)}|1+\delta\cdot|^{1/6}$$
 strongly in $L^2(\R)$, we deduce that
 $$\left|\int_{(-1/2,1/2)}|1+\delta\eta|^{1/6}v(\eta)\,d\eta\right|\ge\epsilon>0,$$
 implying that $v\neq0$. By the Cauchy-Schwarz inequality, we also have
 $$\epsilon\le\norm{v}_{L^2}\left(\int_{(-1/2,1/2)}|1+\delta\eta|^{1/3}\,d\eta\right)^{1/2}\le\norm{v}_{L^2}\left(\int_{(-1/2,1/2)}(1+2|\eta|)^{1/3}\,d\eta\right)^{1/2},$$
 since $\delta\in[0,2]$, which implies the desired lower bound on $\norm{v}_{L^2}$.
\end{proof}

We now turn to the proof of Theorem \ref{thm:refined-stri}, following the strategy of \cite[App. A]{Tao-09} and \cite[Prop. 4.24]{KilVis-book}, that we also followed in \cite{FraLieSab-16}. We first state some properties of dyadic intervals.

\subsection{Dyadic intervals}

\begin{definition}\label{def:dyadic-interval}
An interval $I\subset\R$ is \emph{dyadic} if it can be written as $I=[k,k+1)2^\ell$ with $k\in\Z$ and $\ell\in\Z$. For any interval $I$, we denote by $|I|$ its length and $c(I)$ its center. 
\end{definition}

Any dyadic interval generates two dyadic sub-intervals of half length, and reciprocally any dyadic interval has a unique parent of double length, which is also a dyadic interval. Two dyadic intervals are said to be \emph{adjacent} if they have the same length and share an extremity. 

\begin{definition}
For two dyadic intervals $I$ and $I'$, we write $I\sim I'$ if $I$ and $I'$ are not adjacent, if their parents are not adjacent, but their grand-parents are adjacent (in particular, they have the same length). 
\end{definition}

In the following lemma, we record several useful properties of such intervals.

\begin{lemma}\label{lem:dyadic-intervals}
 Assume $I$ and $I'$ are two dyadic intervals with $I\sim I'$ and either $I,I'\subset\R_+$ or $I,I'\subset\R_-$. Then, the following properties hold:
 \begin{enumerate}
  \item $\forall\eta\in I,\,|\eta|\le2|c(I)|$,
  \item $|c(I')|\le15|c(I)|$,
  \item $\forall\eta\in I,\,\forall\eta'\in I',\, \frac45|c(I+I')|\le|\eta+\eta'|\le\frac65|c(I+I')|$,
  \item $\forall\eta\in I,\,\forall\eta'\in I',\,2|I|\le|\eta-\eta'|\le8|I|$.
 \end{enumerate}
\end{lemma}

\begin{proof}
 Up to replacing $I$ by $-I$ and $I'$ by $-I'$, we may assume that $I,I'\subset\R_+$. We thus may write $I=[k,k+1)|I|$ and $I'=[k',k'+1)|I'|$ for some $k,k'\ge0$. We have $c(I)=(k+1/2)|I|\ge(1/2)|I|$, hence for any $\eta\in I$, we have $\eta\le c(I)+(1/2)|I|\le 2c(I)$, which is (1). Furthermore, since $I\sim I'$, we deduce $k'\le k+7$, and hence $c(I')=(k'+1/2)|I|\le(k+15/2)|I|\le15(k+1/2)|I|=15c(I)$, which is (2). Now let $\eta\in I$, $\eta'\in I'$. Since $I\sim I'$, we have $k+k'\ge4$, hence $c(I+I')=(k+k'+1)|I|\ge5|I|$ which implies $|I|\le(1/5)c(I+I')$. As a consequence 
 $$ \frac45 c(I+I')\le c(I+I')-|I|\le\eta+\eta'\le c(I+I')+|I|\le \frac65 c(I+I'),$$
 which is (3). Finally, (4) follows from the fact that $I\sim I'$. 
\end{proof}

\subsection{Bilinear estimates}

Bilinear estimates are the main building blocks to obtain refined inequalities. In the context of the Stein--Tomas or Strichartz inequalities, they are provided by the deep result of \cite{Tao-03}. Since we work in one space dimension, bilinear estimates are rather easy to obtain by the Hausdorff--Young inequality (as done for instance in \cite[Lem 1.2]{Shao-09b} or \cite[Prop. 2.1]{CarKer-07}). One special feature of our approach is the distinction between positive or negative frequencies, which we may interpret as the separation between the two 'conjugate' points $\xi$ and $-\xi$ which are the main enemies for proving compactness.

For any function $u\in L^2(\R)$ and any interval $I\subset\R$, we define the function $u_I$ by the relation 
$$\hat{u_I}=\1_I\hat{u}.$$

\begin{lemma}\label{lem:bilinear}
 For all $q\ge2$, there exists $C>0$ such that for all dyadic intervals $I\sim I'$ with either $I,I'\subset\R_+$ or $I,I'\subset\R_-$, and for any $u,v\in L^2(\R)$ we have 
 \begin{equation}
  \norm{\left(|D_x|^{1/6}e^{-t\partial_x^3}u_I\right)\left(|D_x|^{1/6}e^{-t\partial_x^3}v_{I'}\right)}_{L^q_{t,x}(\R\times\R)}\le C|c(I)|^{\frac13-\frac1q}|I|^{1-\frac3q}\norm{u}_{L^2(\R)}\norm{v}_{L^2(\R)}.
 \end{equation}
\end{lemma}

\begin{remark}
 The point of the previous lemma is to have $q<3$. Hence, when assuming some properties of the Fourier support, one can do better than the Airy--Strichartz inequality.
\end{remark}

\begin{proof}
 We have the identity for all $x\in\R$
 \begin{multline*}
  \left(|D_x|^{1/6}e^{-t\partial_x^3}u_I\right)(x)\left(|D_x|^{1/6}e^{-t\partial_x^3}v_{I'}\right)(x)\\
  =\frac{1}{2\pi}\int_\R\int_\R e^{ix(\eta+\eta')+it(\eta^3+\eta'^3)}|\eta|^{1/6}|\eta'|^{1/6}\hat{u_I}(\eta)\hat{v_{I'}}(\eta')\,d\eta\,d\eta'.
 \end{multline*}
Denoting by $f(\eta,\eta'):=|\eta|^{1/6}|\eta'|^{1/6}\hat{u_I}(\eta)\hat{v_{I'}}(\eta')$ and changing variables $(r,s)=(\eta+\eta',\eta^3+\eta'^3)=\psi(\eta,\eta')$, we find that
\begin{multline*}
  \left(|D_x|^{1/6}e^{-t\partial_x^3}u_I\right)(x)\left(|D_x|^{1/6}e^{-t\partial_x^3}v_{I'}\right)(x)\\
  =\frac{1}{6\pi}\int_\R\int_\R e^{ixr+its}f(\psi^{-1}(r,s))\frac{dr\,ds}{|\psi^{-1}(r,s)_1^2-\psi^{-1}(r,s)_2^2|},
\end{multline*}
which we estimate using the Hausdorff--Young inequality (where $q'$ is the dual exponent of $q$)
$$\norm{\left(|D_x|^{1/6}e^{-t\partial_x^3}u_I\right)\left(|D_x|^{1/6}e^{-t\partial_x^3}v_{I'}\right)}_{L^q_{t,x}}\le C\norm{\frac{f(\psi^{-1}(r,s))}{|\psi^{-1}(r,s)_1^2-\psi^{-1}(r,s)_2^2|}}_{L^{q'}_{r,s}}.
$$
Undoing the change of variables we find
$$\norm{\frac{f(\psi^{-1}(r,s))}{|\psi^{-1}(r,s)_1^2-\psi^{-1}(r,s)_2^2|}}_{L^{q'}_{r,s}}=\left(\int_\R\int_\R\left|\frac{|\eta|^{1/6}|\eta'|^{1/6}\hat{u_I}(\eta)\hat{v_{I'}}(\eta')}{|\eta^2-\eta'^2|}\right|^{q'}3|\eta^2-\eta'^2|\,d\eta\,d\eta'\right)^{1/q'}.
$$
By Lemma \ref{lem:dyadic-intervals}, we have on the support of the last integral
$$\frac{|\eta|^{q'/6}|\eta'|^{q'/6}}{|\eta+\eta'|^{q'-1}|\eta-\eta'|^{q'-1}}\le C|c(I)|^{1-\frac23 q'}|I|^{1-q'}.$$
Estimating
$$\int_\R|\hat{u_I}(\eta)|^{q'}\,d\eta\le|I|^{1-q'/2}\norm{u}_{L^2}^{q'},$$
we arrive at the result.
\end{proof}

\subsection{Proof of Theorem \ref{thm:refined-stri}}

Let $u\in L^2(\R)$. By splitting $u$ as $u=u_>+u_<$ with $\hat{u_>}=\1_{\R_+}\hat{u}$, we may assume that $\supp\hat{u}\subset\R_+$ or $\supp\hat{u}\subset\R_-$. For any dyadic interval $I$, we use the notation
$$\Psi_I:=|D_x|^{1/6}e^{-t\partial_x^3}u_I.$$
Since for any $(\eta,\eta')\in\R^2$ with $\eta\neq\eta'$, there is a unique pair of dyadic intervals $(I,I')$ of maximal length with $I\sim I'$, $\eta\in I$, and $\eta'\in I'$, the identity
$$\forall\eta\neq\eta',\quad 1=\sum_{I\sim I'}\1_I(\eta)\1_{I'}(\eta')$$
induces the decomposition
\begin{equation}\label{eq:dyadic-decomp}
  \left(|D_x|^{1/6}e^{-t\partial_x^3}u\right)^2=\sum_{I\sim I'}\Psi_I\Psi_{I'},
\end{equation}
which we estimate using the following lemma, exploiting some orthogonality between the Fourier supports of each term under the previous sum.

\begin{lemma}\label{lem:orthogonality}
 There exists $C>0$ such that for all $u\in L^2(\R)$ with ${\rm supp}\,\hat{u}\subset\R_-$ or $\R_+$, we have 
 $$\norm{\sum_{I\sim I'}\Psi_I\Psi_{I'}}_{L^3_{t,x}}^{3/2}\le C\sum_{I\sim I'}\norm{\Psi_I\Psi_{I'}}_{L^3_{t,x}}^{3/2}.$$
\end{lemma}

Assuming Lemma \ref{lem:orthogonality}, we finish the proof of Theorem \ref{thm:refined-stri}. We estimate $\norm{\Psi_I\Psi_{I'}}_{L^3_{t,x}}$ in two different ways. First, using Lemma \ref{lem:bilinear} and Lemma \ref{lem:dyadic-intervals} we conclude that for any $q\in[2,3)$, 
\begin{align}
\norm{\Psi_I\Psi_{I'}}_{L^3_{t,x}} &\le \left(|c(I)|^{-\frac13}|I|^{-1}\norm{\Psi_I\Psi_{I'}}_{L^\ii_{t,x}}\right)^{1-\frac q3}\left(|c(I)|^{\frac1q-\frac13}|I|^{\frac3q-1}\norm{\Psi_I\Psi_{I'}}_{L^q_{t,x}}\right)^{\frac q3}\nonumber \\
&\le C\left(|c(I)|^{-\frac16}|I|^{-\frac12}\norm{\Psi_I}_{L^\ii_{t,x}}\right)^{1-\frac q3}\left(|c(I)|^{-\frac16}|I|^{-\frac12}\norm{\Psi_{I'}}_{L^\ii_{t,x}}\right)^{1-\frac q3}\norm{u}_{L^2}^{\frac{2q}{3}}\nonumber\\
&\le C\left(\sup_{I''\in\cD}|c(I'')|^{-\frac16}|I''|^{-\frac12}\norm{\Psi_{I''}}_{L^\ii_{t,x}}\right)^{2-\frac{2q}{3}}\norm{u}_{L^2}^{\frac{2q}{3}} \,.\label{eq:first-L3-estimate}
\end{align}
Secondly, by an elementary estimate, using also Lemma \ref{lem:dyadic-intervals} we find
$$\norm{\Psi_I\Psi_{I'}}_{L^\ii_{t,x}}\le C |c(I)|^{\frac13}\norm{\hat{u_I}}_{L^1}\norm{\hat{u_{I'}}}_{L^1},$$
and then by interpolation with Lemma \ref{lem:bilinear} we infer that
\begin{equation}\label{eq:second-L3-estimate}
  \norm{\Psi_I\Psi_{I'}}_{L^3_{t,x}}\le C|I|^{1-\frac2s}\norm{\hat{u_I}}_{L^s}\norm{\hat{u_{I'}}}_{L^s}
\end{equation}
for some $3/2\leq s<2$. In conclusion, from identity \eqref{eq:dyadic-decomp}, Lemma \ref{lem:orthogonality}, and estimates \eqref{eq:first-L3-estimate} and \eqref{eq:second-L3-estimate} we deduce that for any $r\le3/2$
\begin{align*}
 \norm{|D_x|^{1/6}e^{-t\partial_x^3}u}_{L^6_{t,x}}^3 & \le C\sum_{I\sim I'}\norm{\Psi_I\Psi_{I'}}_{L^3_{t,x}}^{3/2}\\
 &\le C\sup_{I\sim I'}\norm{\Psi_I\Psi_{I'}}_{L^3_{t,x}}^{3/2-r}\sum_{I\sim I'}\norm{\Psi_I\Psi_{I'}}_{L^3_{t,x}}^r\\
 &\le C\!\left(\sup_{I\in\cD}|c(I)|^{-\frac16}|I|^{-\frac12}\norm{\Psi_{I}}_{L^\ii_{t,x}}\!\right)^{\left(2-\frac{2q}{3}\right)\left(\frac32-r\right)}\!\!\norm{u}_{L^2}^{\frac{2q}{3}\left(\frac32-r\right)}\!\sum_{I\in\cD}\left[|I|^{1-\frac2s}\norm{\hat{u_I}}_{L^s}^2\right]^r\!.
\end{align*}
We now use \cite[Lem. A.3]{FraLieSab-16} (which is itself extracted from \cite[App. A]{Tao-09}) with the choices $\nu=2r/s$ and $\mu=2/s$ and obtain that for $r>1$,
$$\sum_{I\in\cD}\left[|I|^{1-\frac2s}\norm{\hat{u_I}}_{L^s}^2\right]^r\le C\norm{u}_{L^2}^{2r}.$$
This completes the proof of Theorem \ref{thm:refined-stri}.

It thus remains to provide the

\begin{proof}[Proof of Lemma \ref{lem:orthogonality}]
Let us explain the strategy of the proof before giving the details. First, without loss of generality, we may assume that ${\rm supp}\,\hat{u}\subset\R_+$, so that any dyadic interval appearing in the following may be assumed to be included in $\R_+$. We will associate to any interval $J\subset\R_+$ a parallelogram $R(J)\subset\R^2$ such that 
\begin{equation}
\label{eq:fouriersuppinclusion}
{\rm supp}\,\cF_{t,x}\left[\Psi_I\Psi_{I'}\right]\subset R(I+I') \,,
\end{equation}
where $\cF_{t,x}$ denotes the space-time Fourier transform. We then define for a parallelogram $R(J)$ and a number $\alpha>0$ an enlarged parallelogram $(1+\alpha)R(J)$. The main point of the proof will be to show that there is a finite, universal constant such that for any pair $(I,I')$ with $I\sim I'$ the number of pairs $(\tilde I,\tilde I')$ with $\tilde I\sim\tilde I'$ and
\begin{equation}
\label{eq:intersection}
(1+\alpha)R(I+I')\cap (1+\alpha)R(\tilde I + \tilde I') \neq \emptyset
\end{equation}
is bounded by this constant. Once this is shown the conclusion of the lemma follows from \cite[Lem. A.9]{KilVis-book}. 

Let us now carry out the details of this argument. We clearly have 
$${\rm supp}\,\cF_{t,x}\left[\Psi_I\Psi_{I'}\right]\subset\{(\eta^3+\eta'^3,\eta+\eta'),\,\eta\in I,\,\eta'\in I'\},$$
and we will include this last set into a parallelogram in the following fashion. Let $(\omega,\xi)\in\R^2$ with $\omega=\eta^3+\eta'^3$, $\xi=\eta+\eta'$ for some $\eta\in I$, $\eta'\in I'$. Define $c=c(I+I')>0$. A quick computation shows that
$$\omega-\frac14\xi^3=\frac34\xi(\eta-\eta')^2,$$
which we combine with the identity
$$\xi^3=(\xi-c)^3+3(\xi-c)^2c+3(\xi-c)c^2+c^3$$
to infer that
$$\omega-\frac34(\xi-c)c^2-\frac14c^3=\frac34\xi(\eta-\eta')^2+\frac14(\xi-c)^2(\xi+2c).$$
Using Lemma \ref{lem:dyadic-intervals}, we deduce that
$$\frac{12}{5}\le\frac{\omega-\frac34(\xi-c)c^2-\frac14c^3}{c|I|^2}\le\frac{292}{5},$$
which means that we have the inclusion \eqref{eq:fouriersuppinclusion} with the parallelogram
$$R(J)=\left\{(\omega,\xi):\,\xi\in J,\ \frac35\le\frac{\omega-\frac34(\xi-c(J))c(J)^2-\frac14c(J)^3}{c(J)|J|^2}\le\frac{73}{5}\right\}.$$

Let $\alpha>0$ and define $(1+\alpha)R(J)$ to be the $(1+\alpha)$-dilate of the parallelogram $R(J)$, that is, the parallelogram with the same center as $R(J)$ but whose linear part is multiplied by $1+\alpha$. An elementary computation shows that
$$(1+\alpha)R(J)=\!\left\{(\omega,\xi):\,\xi\in (1+\alpha) J,\ \frac35-7\alpha\le\frac{\omega-\frac34(\xi-c(J))c(J)^2-\frac14c(J)^3}{c(J)|J|^2}\le\frac{73}{5}+7\alpha\right\}$$
where $(1+\alpha)J$ denotes the interval with the same center as $J$ but with $1+\alpha$ times its length. We shall now show the universal bound on the number of pairs $(\tilde{I},\tilde{I}')$ with $\tilde{I}\sim\tilde{I}'$ satisfying \eqref{eq:intersection} for some given pair $(I,I')$ with $I\sim I'$. Let us note $J=I+I'$ and $\tilde{J}=\tilde{I}+\tilde{I}'$, and let $(\omega,\xi)\in(1+\alpha)R(J)\cap(1+\alpha)R(\tilde{J})$. By the same reasoning as in the proof of Lemma \ref{lem:dyadic-intervals}-(3), we have $\frac{4-\alpha}{5}c\le\xi\le\frac{6+\alpha}{5}c$, where $c$ denotes either $c(J)$ or $c(\tilde{J})$. In particular, we have
\begin{equation}\label{eq:identification-centers}
\frac{4-\alpha}{6+\alpha}c(J)\le c(\tilde{J})\le\frac{6+\alpha}{4-\alpha}c(J).
\end{equation}
Next, from the relation
$$-3(\xi-c)c^2-c^3+\xi^3=(\xi-c)^2(\xi+2c),$$
we deduce that
$$\omega-\frac14\xi^3=\omega-\frac34(\xi-c)c^2-\frac14c^3-\frac14(\xi-c)^2(\xi+2c),$$
again with $c=c(J)$ or $c=c(\tilde{J})$. In particular,
$$\left(\frac25-\frac{561}{80}\alpha\right)cL^2\le\omega-\frac14\xi^3\le\left(\frac{73}{5}+7\alpha\right)cL^2,$$
where $(c,L)=(c(J),|J|)$ or $(c,L)=(c(\tilde{J}),|\tilde{J}|)$. Choosing $\alpha>0$ small enough such that $\alpha<4$, $\frac35-7\alpha>0$, and $\frac25-\frac{561}{80}\alpha>0$, we deduce that there exist universal numbers $a,b>0$ such that $a|J|\le|\tilde{J}|\le b|J|$. This relation together with \eqref{eq:identification-centers} can be satisfied only for a universal, finite number of pairs $(\tilde I,\tilde I')$ with $\tilde{I}\sim\tilde{I}'$.
\end{proof}

\section{Approximate operators}\label{sec:approximate-operators}

In this section, we provide some properties of the Airy--Strichartz map for functions that concentrate around a frequency. Define the family of operators for any $q\ge2$ and $\delta\in\R$,

$$(T_{p,\delta} u)(t,x)=\frac{1}{\sqrt{2\pi}}\int_\R|1+\delta\xi|^{1/p}e^{ix\xi+it(3\xi^2+\delta\xi^3)}\hat{u}(\xi)\,d\xi.$$

\subsection{Basic estimates}

\begin{lemma}\label{lem:approximate-operator} Let $p>4$ and $q$ such that $2/p+1/q=1/2$. 

 \begin{enumerate}
  \item There exists $C>0$ such that for any $\delta\in\R$ we have
  $$\norm{T_{p,\delta}}_{L^2_x(\R)\to L^p_tL^q_x(\R\times\R)}\le C$$
  \item For any $u\in L^2(\R)$ we have 
  $$T_{p,\delta} u \to e^{-3it\partial_x^2}u$$
  as $\delta\to0$, strongly in $L^p_tL^q_x(\R\times\R)$. 
 \end{enumerate}
\end{lemma}

\begin{proof}
  The first item follows from the Airy--Strichartz inequality \eqref{eq:A-S} by undoing the scaling as in the proof of Lemma \ref{lem:energy-two-bubbles}, which shows that
  $$\norm{T_{p,\delta}}_{L^2_x(\R)\to L^p_tL^q_x(\R\times\R)}=\norm{|D_x|^{1/p}e^{-t\partial_x^3}}_{L^2_x(\R)\to L^p_tL^q_x(\R\times\R)}.$$
  For the second item, using the first item it is enough to prove it for $u$ such that $\hat{u}\in C^\ii_0(\R)$. As in \cite[Proof of Prop. 4.1]{FraLieSab-16}, it is enough to prove the pointwise estimate
  $$|T_{p,\delta} u(t,x)|\le C(t^2+x^2)^{-1/4},$$
  for some $C>0$ independent of $\delta$ and for $\delta$ small enough (both depending on $u$ though). We may write
  $$T_{p,\delta} u(t,x)=\int_\R e^{i\lambda\Phi(\xi)}a(\xi)\,d\xi$$
  with $\lambda=(t^2+x^2)^{1/2}$, $\Phi(\xi)=\omega_1\xi+\omega_2(3\xi^2+\delta\xi^3)$, $(\omega_1,\omega_2)=(x/\lambda,t/\lambda)\in\Sph^2$, $a(\xi)=|1+\delta\xi|^{1/p}\hat{u}(\xi)$. We apply stationary phase estimates. Let $R>0$ such that $\supp\,\hat{u}\subset[-R,R]$. If $\delta\le 1/(2R)$, the function $a$ is $C^\ii$ on $\R$ with all its derivatives uniformly bounded in $\delta$. We have 
  $$\Phi'(\xi)=\omega_1+\omega_2(6\xi+3\delta\xi^2).$$
  Assume $|\omega_2|\le b$, then 
  $$|\Phi'(\xi)|\ge\sqrt{1-b^2}-CbR,$$
  so that for $b=b(R)>0$ small enough, we have $|\Phi'(\xi)|\ge1/2$. Furthermore, all higher derivatives of $\Phi$ are uniformly bounded in $\delta$ due to the $\xi$-localization. Hence, by integration by parts, we get the desired estimate (and even much faster decay) in the region $|\omega_2|\le b$. Let us now consider the case $|\omega_2|>b$. In this case, $\Phi'$ has at most 3 critical points, but 
  $$\Phi''(\xi)=6\omega_2(1+\delta\xi)$$
  so that $|\Phi''(\xi)|\ge3b$, in which case we may apply stationary phase results (again, all higher order derivatives of $\Phi$ are uniformly bounded in $\delta$). This concludes the proof of the second item.
  \end{proof}

\subsection{Local smoothing}

The following result in crucial in order to establish the a.e. convergence (proven in the next subsection) that is used in order to split the $L^p_tL^q_x$-norm in the proof of Theorem \ref{thm:MMM-2}. Interestingly, it needs some assumption about the Fourier support, which is reminiscent of the `resonance' between positive and negative frequencies that is responsible for the presence of the constant $a_p$.

\begin{lemma}\label{lem:local-smoothing-approx}
 Let $p>0$. For any $a\in L^1(\R)$, there exists $C>0$ such that for any $\delta\in\R$ and any $u\in L^2(\R)$ with Fourier transform supported in $\{1+\delta\xi\ge0\}$ we have 
 $$\int_\R\int_\R a(x)\left|f_\delta(-iD_x)T_{p,\delta} u(t,x)\right|^2\,dx\,dt\le C\norm{u}_{L^2}^2,$$
 where
 $$f_\delta(\xi):=\frac{|\xi|^{1/2}}{|1+\delta\xi|^{1/p}}.$$
\end{lemma}

\begin{proof}
In Fourier variables, the integral can be written as 
$$\int_\R\int_\R\hat{a}(\xi'-\xi)f_\delta(\xi)f_\delta(\xi')|1+\delta\xi|^{1/p}|1+\delta\xi'|^{1/p}\hat{u}(\xi)\bar{\hat{u}(\xi')}\delta(3\xi^2+\delta\xi^3-3\xi'^2-\delta\xi'^3)\,d\xi\,d\xi'.$$
 As in the proof of Lemma 4.4 in \cite{FraLieSab-16}, Schur's test implies that it is enough to bound
 \begin{align*}
    \sup_{\substack{\xi:\,1+\delta\xi\ge0 \\ \delta\in\R}}\int_\R|\hat{a}(\xi'-\xi)| & f_\delta(\xi)f_\delta(\xi') |1+\delta\xi|^{1/p}|1+\delta\xi'|^{1/p}\delta(3\xi^2+\delta\xi^3-3\xi'^2-\delta\xi'^3)\,d\xi' \\
    &\le\norm{\hat{a}}_{L^\ii}
    \sup_{\substack{\xi:\,1+\delta\xi\ge0 \\ \delta\in\R}}\max_{\substack{\xi'\neq0:\,1+\delta\xi'\ge0 \\ 3\xi'^2+\delta\xi'^3=3\xi^2+\delta\xi^3}}\frac{|1+\delta\xi|^{1/p}f_\delta(\xi)|1+\delta\xi'|^{1/p}f_\delta(\xi')}{|\xi'||1+\delta\xi'/2|}\\
    &=\norm{\hat{a}}_{L^\ii}\sup_{\substack{\xi:\,1+\delta\xi\ge0 \\ \delta\in\R}}\max_{\substack{\xi'\neq0:\,1+\delta\xi'\ge0 \\ 3\xi'^2+\delta\xi'^3=3\xi^2+\delta\xi^3}}\frac{|\xi|^{1/2}}{|\xi'|^{1/2}|1+\delta\xi'/2|}.
 \end{align*}
 Here, the max is due to the fact that the equation $3\xi'^2+\delta\xi'^3=c$ may have several (at most three) solutions. Because of the constraint $\delta\xi'\ge-1$, we have $1+\delta\xi'/2\ge1/2$, hence we only have to care about the quotient $(|\xi|/|\xi'|)^{1/2}$. When $\delta=0$, the sup is clearly $1$. When $\delta\neq0$, defining $\eta=\delta\xi\ge-1$ and $\eta'=\delta\xi'\ge-1$, we see that the sup is independent of $\delta$ and we have to show that
 $$\sup_{\substack{\eta,\eta':\,\eta,\eta'\ge-1\,\eta'\neq0 \\ 3\eta'^2+\eta'^3=3\eta^2+\eta^3}}\frac{|\eta|}{|\eta'|}<\ii.$$
 Defining $g(x)=3x^2+x^3$, it is elementary to show that $g^{-1}([0,2])\cap[-1,+\ii)=[-1,a]$ for some $a\in(1/2,1)$, and that $g$ is increasing on $(a,\ii)$; so that in the previous sup we may assume $\eta\in[-1,a]$, and $\eta\neq0$ since $\eta'\neq0$. For all $\eta\in[-1,a]\setminus\{0\}$, we have $g^{-1}(\{g(\eta)\})=\{\eta,y_\eta\}$ with $\eta y_\eta<0$. Thus, the above supremum equals
 $$\sup_{\eta\in[-1,a]\setminus\{0\}}\max\left\{1,\frac{|\eta|}{|y_\eta|}\right\}.$$
 In the case $\eta>0$, since $g$ is decreasing on $[-1,0]$ and $g(-\eta)=3\eta^2-\eta^3<3\eta^2+\eta^3=g(\eta)$, we deduce $y_\eta<-\eta$, thus $|y_\eta|>|\eta|$. In the case $\eta<0$, we have $g(-\eta/2)=(3/4)\eta^2-(1/8)\eta^3<3\eta^2+\eta^3=g(\eta)$ since $\eta\ge-1$ and hence we have $y_\eta\ge-\eta/2$ which also implies $|\eta'|\ge(1/2)|\eta|$.
\end{proof}

\subsection{Local convergence}

Here is the almost everywhere convergence result that we used in the proof of Theorem \ref{thm:MMM-2} in order to apply the generalized Br\'ezis--Lieb lemma. Again, the Fourier support condition is crucial.

\begin{lemma}\label{lem:local-convergence}
 Let $p\ge2$. Let $\delta_n\to0$ and $(u_n)\subset L^2(\R)$ a sequence converging weakly to zero in $L^2(\R)$, whose Fourier transform is supported in $\{1+\delta_n\xi\ge0\}$. Then $T_{p,\delta_n}u_n\to0$ in $L^2_{{\rm loc},t,x}(\R\times\R)$, and hence almost everywhere on $\R\times\R$ up to a subsequence.
\end{lemma}

\begin{proof}
 Let $K\subset\R\times\R$ a bounded set and $a>0$ a Schwartz function. We have 
 $$\1_KT_{p,\delta_n}u_n=\1_KT_{p,\delta_n}P_\Lambda u_n+\1_KT_{p,\delta_n}P_\Lambda^\perp u_n,$$
 where $P_\Lambda$ is the Fourier multiplier by $\1_{\{|\xi|\le\Lambda\}}$ for some $\Lambda>0$ and $P_\Lambda^\perp=1-P_\Lambda$. We estimate the second term with the local smoothing estimate of Lemma \ref{lem:local-smoothing-approx}:
 \begin{align*}
  \norm{\1_KT_{p,\delta_n}P_\Lambda^\perp u_n}_{L^2_{t,x}} &\le \sup_K 1/a\norm{aT_{p,\delta_n}f_{\delta_n}(-iD_x)}_{L^2_x\to L^2_{t,x}}\norm{f_{\delta_n}(-iD_x)^{-1}P_\Lambda^\perp}_{L^2_x\to L^2_x}\norm{u_n}_{L^2_x} \\
  &\le C_K\sup_{|\xi|\ge\Lambda}\left(\frac{1}{|\xi|^{p/2}}+\frac{|\delta_n|}{|\xi|^ {p/2-1}}\right)^{1/p},
 \end{align*}
 which can be made less than $\epsilon>0$, for any given $\epsilon>0$, provided that $n$ and $\Lambda$ are large enough (depending only on $\epsilon$). For such a fixed $\Lambda$, we now claim that $\1_KT_{p,\delta_n}P_\Lambda u_n\to0$ strongly in $L^2(\R)$ as $n\to\ii$. Indeed, for any $(t,x)\in\R\times\R$, 
 $$\1_KT_{p,\delta_n}P_\Lambda u_n(t,x)=\1_K(t,x)\int_\R|1+\delta_n\xi|^{1/p}\1(|\xi|\le\Lambda)e^{ix\xi+it(3\xi^2+\delta_n\xi^3)}\hat{u_n}(\xi)\,d\xi\to0$$
 since $u_n\to0$ weakly in $L^2(\R)$ and the function $\xi\mapsto|1+\delta_n\xi|^{1/p}\1(|\xi|\le\Lambda)e^{ix\xi+it(3\xi^2+\delta_n\xi^3)}$ converges strongly in $L^2(\R)$ as $n\to\ii$ (since $(\delta_n)$ converges). Furthermore, we have 
 $$|\1_KT_{p,\delta_n}P_\Lambda u_n(t,x)|\le\1_K(t,x)\Lambda(1+|\delta_n|\Lambda)^{1/p}\norm{u_n}_{L^2_x}\le C\1_K(t,x),$$
 hence the result follows by dominated convergence.
\end{proof}

\appendix

\section{A generalized Br\'ezis--Lieb lemma for mixed Lebesgue spaces}\label{app:BL}

Let us review some basics. We assume that $(X,dx)$ and $(Y,dy)$ are measure spaces and consider a sequence $(f_n)$ of non-negative measurable functions on $X\times Y$ which converges almost everywhere to some function $f$. Moreover, we fix an exponent $r>0$. Our first remark is that the monotone convergence theorem remains true, in the sense that, if for each $n$ one has $f_{n+1}\geq f_n$ almost everywhere, then
$$
\lim_{n\to\infty} \int_Y \left( \int_X f_n\,dx \right)^r dy = \int_Y \left( \int_X f\,dx \right)^r dy \,.
$$
To see this, we first apply for almost every fixed $y\in Y$ the usual monotone convergence theorem in $X$ to see that $g_n := \left( \int_X f_n\,dx \right)^r$ converges to $g := \left( \int_X f\,dx \right)^r$. Indeed, by Fubini's theorem, for a.e. $y\in Y$, $f_n(\cdot,y)$ converges to $f(\cdot,y)$ a.e. on $X$. Then, we apply the monotone convergence theorem in $Y$ to $(g_n)$ and we obtain the claim.

Our second remark is that Fatou's lemma remains true, in the sense that
$$
\liminf_{n\to\infty} \int_Y \left( \int_X f_n\,dx \right)^r dy \geq \int_Y \left( \int_X f\,dx \right)^r dy
$$
This follows, as usual, by applying the monotone convergence theorem to $F_n:=\inf_{m\geq n} f_m$.

Our third remark is that the dominated convergence theorem remains true, in the sense that, if $f_n\leq F$ with $\int_Y \left( \int_X F \right)^r dy <\infty$, then
$$
\lim_{n\to\infty} \int_Y \left( \int_X f_n\,dx \right)^r dy = \int_Y \left( \int_X f\,dx \right)^r dy \,.
$$
To see it, just apply the usual dominated convergence theorem, first to the sequence $f_n(\cdot,y)$ for a.e. $y\in Y$, and then to the sequence $(\int_X f_n\,dx)^r$.

In mixed Lebesgue space, we have the following version of the triangle inequality.

\begin{lemma}\label{lem:triangle-mixed}
 Let $(X,dx)$ and $(Y,dy)$ be measure space, let $0<p,q<\ii$, and let $f,g\in L^p_x L^q_y(X\times Y)$. Then, we have 
 $$\norm{f+g}_{L^p_xL^q_y}^\beta\le\norm{f}_{L^p_xL^q_y}^\beta+\norm{g}_{L^p_xL^q_y}^\beta,$$
 where $\beta=\beta(p,q)=\min(p,q,1)$.
\end{lemma}

\begin{proof}
 Throughout the proof, we will use the inequality $(a+b)^r\le a^r+b^r$ for all $a,b\ge0$, $r\in(0,1]$. We distinguish 4 cases. First, if $p,q\ge1$, then it follows from the triangle inequality in $L^p_xL^q_y$. Secondly, if $p<1\le q$, we have
$$
\norm{f+g}_{L^p_xL^q_y}^p=\int_X\norm{f+g}_{L^q_y}^p\,dx\le\int_X\left(\norm{f}_{L^q_y}+\norm{g}_{L^q_y}\right)^p\,dx\le\int_X\left(\norm{f}_{L^q_y}^p+\norm{g}_{L^q_y}^p\right)\,dx.
$$
Thirdly, if $q<1$ and $p\ge q$, then
\begin{multline*}
  \norm{f+g}_{L^p_xL^q_y}^q=\norm{\int_Y|f+g|^q\,dy}_{L^{p/q}_x}\le\norm{\int_Y|f|^q\,dy+\int_Y|g|^q\,dy}_{L^{p/q}_x}\\
  \le\norm{\int_Y|f|^q\,dy}_{L^{p/q}_x}+\norm{\int_Y|g|^q\,dy}_{L^{p/q}_x}.
\end{multline*}
Finally, if $q<1$ and $p<q$, then
\begin{multline*}
 \norm{f+g}_{L^p_xL^q_y}^p=\int_X\left(\int_Y|f+g|^q\,dy\right)^{p/q}\,dx\le\int_X\left(\int_Y|f|^q\,dy+\int_Y|f|^q\,dy\right)^{p/q}\,dx\\
 \le\int_X\left(\int_Y|f|^q\,dy\right)^{p/q}\,dx+\int_X\left(\int_Y|g|^q\,dy\right)^{p/q}\,dx.
\end{multline*}
\end{proof}

After these preliminaries, we can state and prove the one-sided analogue of the Br\'ezis--Lieb lemma, which is originally due to \cite{Lieb-83b,BreLie-83}. In \cite[Lem. 3.1]{FraLieSab-16} we have obtained a two-fold generalization of this lemma, namely, we allow the leading term to depend on $n$ and we allow for a remainder that converges strongly to zero. The following proposition is a generalization of this generalization to the case of mixed Lebesgue spaces. We emphasize that instead of equality we only have an asymptotic inequality.

\begin{proposition}\label{bl}
Let $(X,dx)$ and $(Y,dy)$ be measure spaces and $(f_n)$ be a sequence of measurable functions on $X\times Y$, and let $0< p, q<\infty$. Assume that
$$
\sup_n \int_X \left( \int_Y |f_n|^q \,dy \right)^{p/q} dx <\infty \,,
$$
and that $f_n$ may be split as 
$$f_n=\pi_n+\rho_n+\sigma_n$$
with $|\pi_n|\le\Pi$ for some $\Pi\in L^p_xL^q_y(X\times Y)$, $\rho_n\to0$ a.e. in $(x,y)\in X\times Y$, and $\sigma_n\to0$ in $L^p_x L^q_y(X \times Y)$. Then, as $n\to\infty$,
$$\norm{f_n}_{L^p_xL^q_y}^\alpha\le\norm{\pi_n}_{L^p_x L^q_y}^\alpha+\norm{\rho_n}_{L^p_x L^q_y}^\alpha+o(1),$$
where $\alpha=\alpha(p,q)=\min(p,q)$.
\end{proposition}

\begin{proof}
 We first show that we may get rid of $\sigma_n$, that is,
\begin{equation}
\label{eq:bl0}
\norm{f_n}_{L^p_x L^q_y}=\norm{\pi_n+\rho_n}_{L^p_x L^q_y}+o_{n\to\ii}(1).
\end{equation}
This follows from Lemma \ref{lem:triangle-mixed}, which implies that with $\beta=\min(\alpha,1)$,
$$\left|\norm{f_n}_{L^p_xL^q_y}^\beta-\norm{\pi_n+\rho_n}_{L^p_xL^q_y}^\beta\right|\le\norm{\sigma_n}_{L^p_xL^q_y}^\beta=o_{n\to\ii}(1) \,.$$
For $\alpha\ge 1$ this immediately gives \eqref{eq:bl0} and for $\alpha<1$ we use in addition the boundedness of $\|f_n\|_{L^p_x L^q_y}$ to deduce \eqref{eq:bl0}.

Next, we shall show that
\begin{equation}
\label{eq:bl}
\int_X \left( \int_Y \left| |\pi_n+\rho_n|^q-|\pi_n|^q -|\rho_n|^q \right| dy \right)^{p/q} dx = o_{n\to\ii}(1) \,.
\end{equation}
Let us first argue that this implies the conclusion. When $p\le q$, we use the elementary inequality
$$
A^\theta \le B^\theta + C^\theta + |A-B-C|^\theta \,,
\qquad A,B,C\ge 0\,,\ 0<\theta\le 1 \,,
$$
with $\theta=p/q$ and $A= \int_Y |\pi_n+\rho_n|^q\,dy$, $B=\int_Y |\pi_n|^q\,dy$ and $C=\int_Y |\rho_n|^q\,dy$. Then
$$
\int_X |A-B-C|^\theta \,dx \leq \int_X \left( \int_Y \left| |\pi_n+\rho_n|^q-|\pi_n|^q -|\rho_n|^q \right| dy \right)^{p/q} dx \,,
$$
so the conclusion follows by integrating the elementary inequality with respect to $x$. In the other case $p>q$, the inequality \eqref{eq:bl} implies that
$$\int_Y|\pi_n+\rho_n|^q\,dy=\int_Y|\pi_n|^q\,dy+\int_Y|\rho_n|^q\,dy+o_{L^{p/q}_x}(1),$$
as $n\to\ii$, so that the result follows from the triangle inequality in $L^{p/q}_x$. Thus, it remains to prove \eqref{eq:bl}. As in the usual Br\'ezis--Lieb proof, we use the fact that for any $\epsilon>0$ there is a $C_\epsilon$ such that for any $a,b\in\C$,
$$
\left| |a+b|^q - |b|^q \right| \leq \epsilon |b|^q + C_\epsilon |a|^q \,.
$$
Let us define
$$
h_n^{(\epsilon)} := \left( \left| |\pi_n+\rho_n|^q - |\pi_n|^q -|\rho_n|^q\right| - \epsilon |\rho_n|^q \right)_+ \,.
$$
On the full measure set $\{\Pi<\ii\}\cap\{\rho_n\to0\}$, $h_n^{(\epsilon)}\to0$ since $\pi_n(x,y)$ is bounded there. Hence, $h_n^{(\epsilon)}\to0$ almost everywhere. Since by the above inequality,
$$
\left| |\pi_n+\rho_n|^q - |\pi_n|^q -|\rho_n|^q\right| \leq \left| |\pi_n+\rho_n|^q - |\rho_n|^q \right| + |\pi_n|^q \leq \epsilon |\rho_n|^q + (1+C_\epsilon) |\pi_n|^q \,,
$$
we have $h_n^{(\epsilon)} \leq (1+C_\epsilon)|\Pi|^q$. Thus, by the analogue of the dominated convergence theorem recalled above,
\begin{equation}
\label{eq:blrem}
\int_X \left( \int_Y h_n^{(\epsilon)}\,dy \right)^{p/q} dx \to 0 \,.
\end{equation}
By definition of $h_n^{(\epsilon)}$ we have
$$
\left| |\pi_n+\rho_n|^q-|\pi_n|^q -|\rho_n|^q \right| \leq \epsilon |\rho_n|^q + h_n^{(\epsilon)}
$$
and therefore
$$
\int_X \left( \int_Y \left| |\pi_n+\rho_n|^q-|\pi_n|^q -|\rho_n|^q \right| dy \right)^{p/q} \,dx \leq \int_X \left( \int_Y \left( \epsilon |\rho_n|^q + h_n^{(\epsilon)}\right) dy \right)^{p/q} \,dx \,.
$$
In this inequality we first take the limsup as $n\to\infty$ and then we let $\epsilon\to 0$. Again by Lemma \ref{lem:triangle-mixed} and \eqref{eq:blrem}, we have
$$\int_X \left( \int_Y \left( \epsilon |\rho_n|^q + h_n^{(\epsilon)}\right) dy \right)^{p/q} \,dx=\epsilon^p\int_X \left( \int_Y |\rho_n|^q  dy \right)^{p/q} \,dx+o_{n\to\ii}(1).$$
Since the $L^p_xL^q_y$-norm of $f_n$, $\pi_n$, and $\sigma_n$ are uniformly bounded in $n$, the $L^p_xL^q_y$-norm of $\rho_n$ is uniformly bounded in $n$ by Lemma \ref{lem:triangle-mixed}. This proves \eqref{eq:bl}.
\end{proof}

\section{A homogenization result in mixed Lebesgue spaces}

The following is an extension of \cite[Lem. 5.2]{Allaire-92} to mixed Lebesgue spaces. We use the convention for the torus
$$\T^k:=(\R/(2\pi\Z))^k$$
and denote by $d\theta$ normalized Lebesgue measure on $\T^k$.

\begin{lemma}\label{lem:homogenization-mixed}
 Let $r>0$, $M,N\in\N^*$, and 
 $$\psi:\R^M\times\R^N\times\T^M\times\T^N\to\R_+$$ 
 a function satisfying the following assumptions:
 there exists a zero measure set $E\subset\R^M\times\R^N$ such that
 \begin{enumerate}
  \item For any $(x_1,x_2)\notin E$, $(\theta_1,\theta_2)\mapsto\psi(x_1,x_2,\theta_1,\theta_2)$ is continuous on $\T^M\times\T^N$;
  \item For any $(\theta_1,\theta_2)\in \T^M\times\T^N$, $(x_1,x_2)\mapsto\psi(x_1,x_2,\theta_1,\theta_2)$ is measurable on $\R^M\times\R^N$;
  \item $$\int_{\R^M}\left(\int_{\R^N}\sup_{(\theta_1,\theta_2)\in \T^M\times\T^N}\psi(x_1,x_2,\theta_1,\theta_2)\,dx_2\right)^r\,dx_1<\ii.$$
 \end{enumerate}
 Then, we have
 \begin{multline*}
    \lim_{\epsilon\to0}\int_{\R^M}\left(\int_{\R^N} \psi(x_1,x_2,x_1/\epsilon^2,x_2/\epsilon)\,dx_2\right)^r\,dx_1\\
    =\int_{\T^M}\int_{\R^M}\left(\int_{\T^N}\int_{\R^N}\psi(x_1,x_2,\theta_1,\theta_2)\,dx_2\,d\theta_2\right)^r\,dx_1\,d\theta_1 \,.
 \end{multline*}
\end{lemma}

\begin{remark}
We state the lemma with the scale $\epsilon^2$ for $x_1$ only for our application; it can be replaced by any scale of the form $f(\epsilon)$ with $f(\epsilon)\to0$ as $\epsilon\to0$.
\end{remark}

\begin{proof}
 We mimic the proof of \cite[Lem. 5.2]{Allaire-92}, adapting it to the context of mixed Lebesgue spaces. Notice that our assumptions imply that $\psi$ is of Carath\'eodory type \cite[Def. VIII.1.2]{EkeTem-99} so that with the help of Fubini's theorem, all the integrals that we consider are well-defined (the measurability is the hard part; however by \cite[Prop. VIII.1.1]{EkeTem-99} the function $\psi$ coincides with a measurable function on $\R^M\times\R^N\times\T^M\times\T^N$ a.e. in $\R^M\times\R^N$, which imply that all the functions we consider are measurable on the appropriate space). Let $(Y_i)$ a paving of $\T^M$ by disjoint cubes of side length $1/n$. We first prove the result for the function 
 $$\psi_n(x_1,x_2,\theta_1,\theta_2):=\sum_i\psi(x_1,x_2,y_i,\theta_2)\1_{Y_i}(\theta_1),$$
 where $(y_i)$ are arbitrary points in $Y_i$. We have
 \begin{multline*}
    \int_{\R^M}\left(\int_{\R^N} \psi_n(x_1,x_2,x_1/\epsilon^2,x_2/\epsilon)\,dx_2\right)^r\,dx_1\\
    =\sum_i\int_{\R^M}\1_{Y_i}(x_1/\epsilon^2)\left(\int_{\R^N}\psi(x_1,x_2,y_i,x_2/\epsilon)\,dx_2\right)^r\,dx_1.
 \end{multline*}
 Using Fubini's theorem and \cite[Lem. 5.2]{Allaire-92}, we have for a.e. $x_1\in\R^M$ that
 $$\lim_{\epsilon\to0}\int_{\R^N}\psi(x_1,x_2,y_i,x_2/\epsilon)\,dx_2=\int_{\T^N}\int_{\R^N}\psi(x_1,x_2,y_i,\theta_2)\,dx_2\,d\theta_2.$$
 Furthermore, we have the uniform bound
 \begin{multline*}
  \left|\left(\int_{\R^N}\psi(x_1,x_2,y_i,x_2/\epsilon)\,dx_2\right)^r-\left(\int_{\T^N}\int_{\R^N}\psi(x_1,x_2,y_i,\theta_2)\,dx_2\,d\theta_2\right)^r\right|\\
  \le2\left(\int_{\R^N}\sup_{(\theta_1,\theta_2)\in \T^M\times\T^N}\psi(x_1,x_2,\theta_1,\theta_2)\,dx_2\right)^r
 \end{multline*}
 which is integrable on $\R^M$ by assumption. By Lebesgue's dominated convergence theorem, we deduce that
 \begin{multline*}
    \int_{\R^M}\left(\int_{\R^N} \psi_n(x_1,x_2,x_1/\epsilon^2,x_2/\epsilon)\,dx_2\right)^r\,dx_1\\
    =\sum_i\int_{\R^M}\1_{Y_i}(x_1/\epsilon^2)\left(\int_{\T^N}\int_{\R^N}\psi(x_1,x_2,y_i,\theta_2)\,dx_2\,d\theta_2\right)^r\,dx_1+o_{\epsilon\to0}(1).
 \end{multline*}
Applying \cite[Lem. 5.2]{Allaire-92} to the function
$$
(x_1,\theta_1) \mapsto \1_{Y_i}(\theta_1) \left(\int_{\T^N}\int_{\R^N}\psi(x_1,x_2,y_i,\theta_2)\,dx_2\,d\theta_2\right)^r
$$
we obtain
 \begin{align*}
    & \int_{\R^M}\left(\int_{\R^N} \psi_n(x_1,x_2,x_1/\epsilon^2,x_2/\epsilon)\,dx_2\right)^r\,dx_1\\
    &\qquad =\sum_i \int_{\T^M} \int_{\R^M}\1_{Y_i}(\theta_1)\left(\int_{\T^N}\int_{\R^N}\psi(x_1,x_2,y_i,\theta_2)\,dx_2\,d\theta_2\right)^r\,dx_1\,d\theta_1 +o_{\epsilon\to0}(1) \\
    &\qquad = \int_{\T^M} \int_{\R^M} \left(\int_{\T^N}\int_{\R^N}\psi_n(x_1,x_2,\theta_1,\theta_2)\,dx_2\,d\theta_2\right)^r\,dx_1\,d\theta_1 +o_{\epsilon\to0}(1) \,,
 \end{align*}
which is the claimed formula for $\psi_n$ instead of $\psi$.
 
In the remainder of the proof we derive the formula for $\psi$ by showing that $\psi_n$ approximates $\psi$ in a suitable topology. Indeed, as in \cite[Lem. 5.2]{Allaire-92}, we know that the function
 $$\delta_n(x_1,x_2):=\sup_{(\theta_1,\theta_2)\in\T^M\times\T^N}|\psi_n(x_1,x_2,\theta_1,\theta_2)-\psi(x_1,x_2,\theta_1,\theta_2)|$$
 satisfies $\delta_n\to0$ a.e. in $(x_1,x_2)$ and that $0\le\delta_n(x_1,x_2)\le g(x_1,x_2)$ with $$\begin{cases}
 \dps g(x_1,x_2):=2\sup_{(\theta_1,\theta_2)\in\T^M\times\T^N}\psi(x_1,x_2,\theta_1,\theta_2),\\
 \dps \int_{\R^M}\left(\int_{\R^N}g(x_1,x_2)\,dx_2\right)^r\,dx_1<\ii.                                                                                               \end{cases}
 $$
 Again by Fubini's theorem and Lebesgue's dominated convergence theorem, we deduce that 
 $$\lim_{n\to\ii}\int_{\R^M}\left(\int_{\R^N}\delta_n(x_1,x_2)\,dx_2\right)^r\,dx_1=0.$$
 For shortness, let us introduce the notations
 $$I_{1,2,\epsilon}[\psi]:=\int_{\R^M}\left(\int_{\R^N} \psi(x_1,x_2,x_1/\epsilon^2,x_2/\epsilon)\,dx_2\right)^r\,dx_1,$$
 $$\bar{I_{1,2}}[\psi]:=\int_{\T^M}\int_{\R^M}\left(\int_{\T^N}\int_{\R^N}\psi(x_1,x_2,\theta_1,\theta_2)\,dx_2\,d\theta_2\right)^r\,dx_1\,d\theta_1.$$
We need to show that $I_{1,2,\epsilon}[\psi]\to \bar{I_{1,2}}[\psi]$ as $\epsilon\to 0$ and, to do so, we distinguish whether $r\le 1$ or $r>1$. 
 
 If $r\le1$, we split
 $$I_{1,2,\epsilon}[\psi]-\bar{I_{1,2}}[\psi]=I_{1,2,\epsilon}[\psi]-I_{1,2,\epsilon}[\psi_n]+I_{1,2,\epsilon}[\psi_n]-\bar{I_{1,2}}[\psi_n]+\bar{I_{1,2}}[\psi_n]-\bar{I_{1,2}}[\psi].$$
 Using $|a^r-b^r|\le |a-b|^r$, we deduce that
 $$|I_{1,2,\epsilon}[\psi]-I_{1,2,\epsilon}[\psi_n]|\le \norm{\delta_n}_{L^rL^1}^r \,,\qquad |\bar{I_{1,2}}[\psi_n]-\bar{I_{1,2}}[\psi]|\le \norm{\delta_n}_{L^rL^1}^r,$$
 so that for all $\alpha>0$, there is $n$ large enough so that for all $\epsilon>0$,
 $$|I_{1,2,\epsilon}[\psi]-\bar{I_{1,2}}[\psi]|\le|I_{1,2,\epsilon}[\psi_n]-\bar{I_{1,2}}[\psi_n]|+\alpha.$$
 Taking the limit $\epsilon\to0$, we find the desired result.

 If $r>1$, we introduce the notation
 $$I_{2,\epsilon}[\psi](x_1):=\int_{\R^N} \psi(x_1,x_2,x_1/\epsilon^2,x_2/\epsilon)\,dx_2,$$
 $$\bar{I_2}[\psi](\theta_1,x_1):=\int_{\T^N}\int_{\R^N}\psi(x_1,x_2,\theta_1,\theta_2)\,dx_2\,d\theta_2,$$
 so that $I_{1,2,\epsilon}[\psi]^{1/r}=\norm{I_{2,\epsilon}[\psi]}_{L^r_{x_1}}$ and $\bar{I_{1,2}}[\psi]^{1/r}=\norm{\bar{I_2}[\psi]}_{L^r_{\theta_1,x_1}}$. We now split
 \begin{multline*}
    I_{1,2,\epsilon}[\psi]^{1/r}-\bar{I_{1,2}}[\psi]^{1/r}=\norm{I_{2,\epsilon}[\psi]}_{L^r_{x_1}}-\norm{I_{2,\epsilon}[\psi_n]}_{L^r_{x_1}}+\norm{I_{2,\epsilon}[\psi_n]}_{L^r_{x_1}}-\norm{\bar{I_2}[\psi_n]}_{L^r_{\theta_1,x_1}}\\
    +\norm{\bar{I_2}[\psi_n]}_{L^r_{\theta_1,x_1}}-\norm{\bar{I_2}[\psi]}_{L^r_{\theta_1,x_1}}.
 \end{multline*}
 We now use the estimates
 $$\left|\norm{I_{2,\epsilon}[\psi]}_{L^r_{x_1}}-\norm{I_{2,\epsilon}[\psi_n]}_{L^r_{x_1}}\right|\le\norm{I_{2,\epsilon}[\psi-\psi_n]}_{L^r_{x_1}}\le\norm{\delta_n}_{L^rL^1},$$
 $$\left|\norm{\bar{I_2}[\psi_n]}_{L^r_{\theta_1,x_1}}-\norm{\bar{I_2}[\psi]}_{L^r_{\theta_1,x_1}}\right|\le\norm{\bar{I_2}[\psi_n-\psi]}_{L^r_{\theta_1,x_1}}\le\norm{\delta_n}_{L^rL^1},$$
 to deduce similarly as for $r\leq 1$ that $I_{1,2,\epsilon}[\psi]^{1/r} \to \bar{I_{1,2}}[\psi]^{1/r}$ as $\epsilon\to 0$. This completes the proof of the lemma.
 \end{proof}


\section{A complex interpolation result}


\begin{proposition}\label{prop:complex-interpolation}
Let $1<p_0,p_1,q_0,q_1<\ii$, $\alpha_0,\alpha_1>0$ and $\theta\in(0,1)$. Define 
$$\frac{1}{p_\theta}=\frac{\theta}{p_1}+\frac{1-\theta}{p_0},\quad\frac{1}{q_\theta}=\frac{\theta}{q_1}+\frac{1-\theta}{q_0},\quad\alpha_\theta=\theta\alpha_1+(1-\theta)\alpha_0.$$
Then, there exists $C>0$ such that for all $f:\R_t\times\R_x\to\C$ such that the right side is well-defined, we have 
$$\norm{|D_x|^{\alpha_\theta}f}_{L^{p_\theta}_tL^{q_\theta}_x}\le C\norm{|D_x|^{\alpha_1}f}_{L^{p_1}_tL^{q_1}_x}^\theta\norm{|D_x|^{\alpha_0}f}_{L^{p_0}_tL^{q_0}_x}^{1-\theta}.$$ 
\end{proposition}

\begin{proof}
 By density, it suffices to prove the inequality for any $f$ such that $\hat{f}\in C^\ii_0(\R^2\setminus\{(t,0),\,t\in\R\})$, where $\hat{f}$ is the $x$-Fourier transform of $f$. By duality, it is enough to prove that there exists $C>0$ such that for all $g\in L^{p_\theta'}L^{q_\theta'}$ we have
 \begin{equation}\label{eq:complex-interp-obj}
    \langle g,|D_x|^{\alpha_\theta}f\rangle\le C\norm{g}_{L^{p_\theta'}L^{q_\theta'}}\norm{|D_x|^{\alpha_1}f}_{L^{p_1}_tL^{q_1}_x}^\theta\norm{|D_x|^{\alpha_0}f}_{L^{p_0}_tL^{q_0}_x}^{1-\theta}.
 \end{equation}
 Hence, let $g\in L^{p_\theta'}L^{q_\theta'}$. We write $g$ as $g=|g|h$ with $|g|$ and $h$ measurable, $|h|\le1$. For $z\in\C$, consider the function
 $$\phi(z)=(1+z)^{-1}\langle g_z,|D_x|^{z\alpha_1  + (1-z)\alpha_0}f\rangle,\quad g_z(t,x):=\norm{g(t,\cdot)}_{L^{q_\theta'}}^{cz+d}|g(t,x)|^{az+b}h(t,x),$$
 with the convention $0^z:=0$, and where the parameters $a,b,c,d$ are chosen so that
 $$a=\frac{q_0'-q_1'}{\theta q_0'+(1-\theta)q_1'},\quad  b=\frac{q_\theta'}{q_0'}=\frac{q_1'}{\theta q_0'+(1-\theta)q_1'},\quad
    c=-\frac{d}{\theta},\quad
    d=\frac{p_\theta'}{p_0'}-\frac{q_\theta'}{q_0'}.
 $$
 These assumptions imply the relations
 $$a+b=\frac{q_\theta'}{q_1'},\quad b+d=\frac{p_\theta'}{p_0'},\quad a+b+c+d=\frac{p_\theta'}{p_1'}.$$
 Let $S=\{\lambda+is,\,\lambda\in(0,1),\,s\in\R\}$ a strip in the complex plane, and let us show that $\phi$ is analytic on $S$, continuous on $\bar{S}$. For a.e. $(t,x)$, the function $z\mapsto \bar{g_z(t,x)}(|D_x|^{z\alpha_1+(1-z)\alpha_0}f)(t,x)$ is analytic on $S$, continuous on $\bar{S}$ by the support assumptions made on $\hat{f}$. They also imply that there exists a $T>0$ such that for any $N\in\N$, there exists $C_{N,f}$ such that for any $z=\lambda+is\in\bar{S}$ and for any $(t,x)\in\R^2$ we have
 $$|(|D_x|^{z\alpha_1+(1-z)\alpha_0}f)(t,x)|\le C_{N,f}e^{|s|}\1_{[-T,T]}(t)(1+|x|)^{-N}.$$
 This can be done by integration by parts in the $x$-Fourier variables, the factor $e^{|s|}$ coming from the derivatives of $|\xi|^{z\alpha_1+(1-z)\alpha_2}$ which can be bounded by $|s|^M$ for some power $M$ depending on $N$, which we choose to bound independently by $e^{|s|}$. Furthermore, for any $z\in\bar{S}$, the extremal values of $a\lambda+b$ are $b=q_\theta'/q_0'$ and $a+b=q_\theta'/q_1'$ so that
 $$|g(t,x)|^{a\lambda+b}\le|g(t,x)|^{q_\theta'/q_0'}+|g(t,x)|^{q_\theta'/q_1'}\in L^{q_0'}_x+L^{q_1'}_x.$$
 As a consequence, we infer that for a.e. $t\in\R$,
 $$z\mapsto\int_\R \bar{g_z(t,x)}(|D_x|^{z\alpha_1+(1-z)\alpha_0}f)(t,x)\,dx$$
 is analytic on $S$ and continuous on $\bar{S}$. It satisfies the bound for any $z\in\bar{S}$ and a.e. $t\in\R$
 $$\left|\int_\R \bar{g_z(t,x)}(|D_x|^{z\alpha_1+(1-z)\alpha_0}f)(t,x)\,dx\right|\le C_{N,f}\1_{[-T,T]}(t)e^{|s|}\norm{g(t,\cdot)}_{L^{q_\theta'}}^{(a+c)\lambda+b+d}.$$
 The extremal values of $(a+c)\lambda+b+d$ are $b+d=p_\theta'/p_0'$ and $a+b+c+d=p_\theta'/p_1'$ so that
 $$\norm{g(t,\cdot)}_{L^{q_\theta'}}^{(a+c)\lambda+b+d}\le\norm{g(t,\cdot)}_{L^{q_\theta'}}^{p_\theta'/p_0'}+\norm{g(t,\cdot)}_{L^{q_\theta'}}^{p_\theta'/p_1'}\in L^{p_0'}_t+L^{p_1'}_t.$$
 This implies that $\phi$ is analytic on $S$ and continuous on $\bar{S}$, with the bound valid for any $z=\lambda+is\in\bar{S}$,
 $$|\phi(z)|\le C_{N,f}e^{|s|}\norm{g}_{L^{p_\theta'}_tL^{q_\theta'}_x}.$$
 On the boundary of $S$, let us show more precise bounds. For any $s\in\R$, we have 
 $$\left|\int_\R \bar{g_{is}(t,x)}(|D_x|^{\alpha_0+is(\alpha_1-\alpha_0)}f)(t,x)\,dx\right|\le\norm{g_{is}(t,\cdot)}_{L^{q_0'}_x}\norm{|D_x|^{is(\alpha_1-\alpha_0)}|D_x|^{\alpha_0}f}_{L^{q_0}_x}.$$
 For any $\eta\in\R$, the Fourier multiplier by $m_\eta(\xi):=|\xi|^{i\eta}$ satisfies the bounds
 $$|m_\eta(\xi)|\le1,\quad|\xi||m_\eta'(\xi)|\le|\eta|.$$
 By the Marcinkiewicz multiplier theorem \cite[Thm. 5.2.2]{Grafakos-book-08}, this implies the $L^p$-bound for all $p>1$:
 $$\norm{|D_x|^{i\eta}g}_{L^p}\le C(1+|\eta|)\norm{g}_{L^p},$$
 which in our case gives 
 $$\norm{|D_x|^{is(\alpha_1-\alpha_0)}|D_x|^{\alpha_0}f}_{L^{q_0}_x}\le C(1+|s|)\norm{|D_x|^{\alpha_0}f}_{L^{q_0}_x}.$$
 Using the bound
 $$\norm{g_{is}(t,\cdot)}_{L^{q_0'}_x}\le\norm{g(t,\cdot)}_{L^{q_\theta'}}^{b+d},$$
 together with the relation $b+d=p_\theta'/p_0'$, we deduce that
 $$|\phi(is)|\le C\norm{g}_{L^{p_\theta'}_tL^{q_\theta'}_x}\norm{|D_x|^{\alpha_0}f}_{L^{p_0}_tL^{q_0}_x}.$$
 Here, we see the role of the prefactor $(1+z)^{-1}$ in front of $\phi(z)$ to compensate the growth of the $L^p$-multiplier norm of $m_\eta$. By the same method, we have the estimate for all $s\in\R$
 $$|\phi(is)|\le C\norm{g}_{L^{p_\theta'}_tL^{q_\theta'}_x}\norm{|D_x|^{\alpha_1}f}_{L^{p_1}_tL^{q_1}_x}$$
 using the relations $a+b=q_\theta'/q_1'$ and $a+b+c+d=p_\theta'/p_1'$. Using Hadamard's three line lemma \cite[Thm. 5.2.1]{Simon-basic-complex}, we deduce \eqref{eq:complex-interp-obj}, which ends the proof.
\end{proof}


\section{Weak compactness of the Airy--Strichartz map}

\begin{lemma}\label{lem:compactness-smoothing}
 Let $\alpha\in(-1/2,1)$ and $(u_n)\subset L^2(\R)$ a sequence converging weakly to zero in $L^2(\R)$. Then, up to a subsequence, $|D_x|^\alpha e^{-t\partial_x^3}u_n\to0$ a.e. in $\R^2$.
\end{lemma}

The proof follows from some local smoothing properties of the Airy kernel:

\begin{lemma}\label{lem:local-smoothing-airy}
 Let $a\in L^1(\R)$ a non-negative function. Then, for all $u\in L^2(\R)$ we have
 $$\int_{\R^2}a(x)\left||D_x|e^{-t\partial_x^3}u(x)\right|^2dx\,dt\le \frac{1}{3}\norm{a}_{L^1}\norm{u}_{L^2}^2.$$
\end{lemma}

\begin{proof}
By the Plancherel identity, we have
$$\int_{\R^2}a(x)\left||D_x|e^{-t\partial_x^3}u(x)\right|^2dx\,dt=\sqrt{2\pi}\int_{\R^2}\hat{a}(\xi'-\xi)|\xi||\xi'|\hat{u}(\xi)\bar{\hat{u}(\xi')}\delta(\xi^3-\xi'^3)\,d\xi\,d\xi'.$$
Using $\delta(\xi^3-\xi'^3)=\delta(\xi-\xi')/(3\xi^2)$, we deduce
$$\int_{\R^2}a(x)\left||D_x|e^{-t\partial_x^3}u(x)\right|^2dx\,dt=\frac{\sqrt{2\pi}}{3}\int_{\R^2}\hat{a}(0)|\hat{u}(\xi)|^2\,d\xi\le\frac{1}{3}\norm{a}_{L^1}\norm{u}_{L^2}^2$$
\end{proof}

\begin{proof}[Proof of Lemma \ref{lem:compactness-smoothing}] 
 We prove that $|D_x|^\alpha e^{-t\partial_x^3}u_n\to0$ in $L^2_\text{loc}(\R^2)$, which implies the result. Hence, let $K\subset\R^2$ a bounded set, and let us show that $\chi_K|D_x|^\alpha e^{-t\partial_x^3}u_n\to0$ in $L^2(\R^2)$. To this end, let $\epsilon>0$ and $\Lambda>0$. Define $P_\Lambda$ the Fourier multiplier on $L^2_x(\R)$ by $\1(|\xi|\le\Lambda)$, and $P_\Lambda^\perp:=1-P_\Lambda$. We split $\chi_K|D_x|^\alpha e^{-t\partial_x^3}u_n=\chi_KP_\Lambda|D_x|^\alpha e^{-t\partial_x^3}u_n+\chi_KP_\Lambda^\perp|D_x|^\alpha e^{-t\partial_x^3}u_n$, and notice that
 \begin{multline*}
    \norm{\chi_KP_\Lambda^\perp|D_x|^\alpha e^{-t\partial_x^3}u_n}_{L^2_{t,x}}\le\norm{\chi_K e^{x^2}}_{L^\ii_{t,x}}\norm{e^{-x^2}|D_x|e^{-t\partial_x^3}}_{L^2_x\to L^2_{t,x}}\norm{P_\Lambda^\perp|D_x|^{\alpha-1}}_{L^2_x\to L^2_x}\norm{u_n}_{L^2}\\
    \le C_K\Lambda^{\alpha-1},
 \end{multline*}
 for some constant $C_K>0$ independent of $n$, by Lemma \ref{lem:compactness-smoothing} and the boundedness of $(u_n)$ in $L^2(\R)$. Hence, for $\Lambda$ large enough independent of $n$, we have 
 $$\norm{\chi_KP_\Lambda^\perp|D_x|^\alpha e^{-t\partial_x^3}u_n}_{L^2_{t,x}}\le\epsilon.$$ For any fixed $t\in\R$, the operator $\chi_K(t,\cdot)P_\Lambda|D_x|^\alpha e^{-t\partial_x^3}$ is compact on $L^2_x(\R)$, hence 
 $$\chi_K(t,\cdot)P_\Lambda|D_x|^\alpha e^{-t\partial_x^3}u_n\to0$$
 strongly in $L^2_x(\R)$ as $n\to\ii$, by weak convergence of $(u_n)$ in $L^2(\R)$. Furthermore, we always have
 $$\norm{\chi_K(t,\cdot)|D_x|^\alpha P_\Lambda e^{-t\partial_x^3}u_n}_{L^2_x}^2\le C\Lambda^{\alpha+\frac12}\norm{\chi_K(t,\cdot)}_{L^2_x}^2,$$
 with $C>0$ independent of $n$. By Lebesgue's dominated convergence theorem, we deduce that $\chi_KP_\Lambda|D_x|^\alpha e^{-t\partial_x^3}u_n\to0$ in $L^2(\R^2)$ as $n\to\ii$, from which the result follows. 
\end{proof}

\section{Maximizers in the subcritical case}\label{sec:subcritical}

In the subcritical case $\gamma<1/p$, the existence of maximizers is simpler and unconditional. Define 
\begin{equation}
 \cA_{\gamma,p}:=\sup_{u\neq0}\frac{\dps\int_\R\left(\int_\R\left||D_x|^\gamma(e^{-t\partial_x^3}u)(x)\right|^qdx\right)^{p/q}dt}{\norm{u}_{L^2}^p}
\end{equation}
where $q$ is determined by $p$ and $\gamma$ as in \eqref{eq:constraint-exponents}. Then, we can prove the following result.

\begin{theorem}\label{thm:subcrit}
 Let $p>4$, $-1/2<\gamma<1/p$, and $q$ such that $-\gamma+3/p+1/q=1/2$. Then, any maximizing sequence for $\cA_{\gamma,p}$ is precompact up to symmetries and, in particular, there exists maximizers for $\cA_{\gamma,p}$. 
\end{theorem}

This result with $p=q=8$ is due to \cite{HunSha-12}.

\begin{remark}
 The same result holds for real-valued functions, with the same proof.
\end{remark}

\begin{proof}[Proof of Theorem \ref{thm:subcrit}]
 We mimic the proof in Section \ref{sec:MMM}. The analogue of Proposition \ref{prop:MMM} is valid, with the same proof using Lemma \ref{lem:compactness-smoothing} and the condition $\gamma>-1/2$. We now show that $\cA_{\gamma,p}^*=0$, from which the result follows. To do so, we argue by contradiction and assume $\cA_{\gamma,p}^*>0$, and let $(u_n)$ a sequence such that $\norm{u_n}_{L^2}=1$, $u_n\rightharpoonup_{\text{sym}}0$, and 
 $$\limsup_{n\to\ii}\norm{|D_x|^\gamma e^{-t\partial_x^3}u_n}_{L^p_tL^q_x}^p\ge\frac12\cA^*_{\gamma,p}.$$
 In particular, we have 
 $|D_x|^\gamma e^{-t\partial_x^3}u_n\not\rightarrow0$ in $L^p_t L^q_x$. In the subcritical case, we have results identical to Corollary \ref{coro:refined-pq} and Corollary \ref{coro:refined} by interpolating $(\gamma,p,q)$ between $(\tilde{\gamma},p,\tilde{q})$ and $(1/p,p,2p/(p-4))$ with $-1/2<\tilde{\gamma}<\gamma$ and using Proposition \ref{prop:complex-interpolation}. Hence, there exists $(g_n)\subset G$ and $(\eta_n)\subset\R$ with $|\eta_n|\ge1/2$ such that $(\hat{g_nu_n}(\cdot+\eta_n))$ has a non-zero weak limit $v$ in $L^2$ (here, we do not need to distinguish between positive and negative frequencies), with a lower bound
 $$\norm{v}_{L^2}\ge\epsilon>0,$$
 where $\epsilon$ only depends on $p,\gamma$. Again, we must have $|\eta_n|\to\ii$. Writing again $\delta_n:=1/\eta_n\to0$ and $\hat{g_nu_n}(\cdot+\eta_n)=\hat{v}+\hat{r_n}$, we have
 $$\norm{|D_x|^\gamma e^{-t\partial_x^3}u_n}_{L^p_tL^q_x}=|\delta_n|^{\frac1p-\gamma}\norm{T_{\gamma,\delta_n}v+T_{\gamma,\delta_n}r_n}_{L^p_tL^q_x},$$
 where the approximate operator $T_{\gamma,\delta}$ is 
 $$(T_{\gamma,\delta}u)(x):=\frac{1}{\sqrt{2\pi}}\int_\R|1+\delta\xi|^\gamma e^{ix\xi+it(3\xi^2+\delta\xi^3)}\hat{u}(\xi)\,d\xi.$$
 As in Lemma \ref{lem:approximate-operator}, we have
 $$\norm{T_{\gamma,\delta}}_{L^2_x\to L^p_t L^q_x}=C\delta^{\gamma-\frac1p},$$
 and for all $u\in L^2_x(\R)$ and all $\gamma<1/p$,
 $$\lim_{\delta\to0}|\delta|^{\frac1p-\gamma}\norm{T_{\gamma,\delta}u}_{L^p_tL^q_x}=0.$$
 As a consequence, we find that 
 $$\norm{|D_x|^\gamma e^{-t\partial_x^3}u_n}_{L^p_tL^q_x}=|\delta_n|^{\frac1p-\gamma}\norm{T_{\gamma,\delta_n}r_n}_{L^p_tL^q_x}+o_{n\to\ii}(1),$$
 and undoing the change of variables shows that 
 \begin{equation}\label{eq:MMM-subcrit}
    \norm{|D_x|^\gamma e^{-t\partial_x^3}u_n}_{L^p_tL^q_x}^p=\norm{|D_x|^\gamma e^{-t\partial_x^3}w_n}_{L^p_tL^q_x}^p+o_{n\to\ii}(1),
 \end{equation}
 where $\hat{w_n}=\hat{r_n}(\cdot-\eta_n)$. By weak convergence of $r_n$ to zero, we know that
 $\norm{r_n}_{L^2}^2\to1-\norm{v}_{L^2}^2$. Hence, as in the proof of Theorem \ref{thm:MMM-2}, we have $w_n\rightharpoonup_\text{sym}0$ and 
 $$\limsup_{n\to\ii}\norm{|D_x|^\gamma e^{-t\partial_x^3}w_n}_{L^p_tL^q_x}^p\le\cA_{\gamma,p}^*(1-\norm{v}_{L^2}^2)^{p/2},$$
 which we insert in \eqref{eq:MMM-subcrit} to obtain
 $$\limsup_{n\to\ii}\norm{|D_x|^\gamma e^{-t\partial_x^3}u_n}_{L^p_tL^q_x}^p\le\cA_{\gamma,p}^*(1-\norm{v}_{L^2}^2)^{p/2}\le\cA_{\gamma,p}^*(1-\epsilon^2)^{p/2}.$$
 Taking the supremum over all such sequences $(u_n)$, we find
 $$\cA^*_{\gamma,p}\le\cA_{\gamma,p}^*(1-\epsilon^2)^{p/2}<\cA^*_{\gamma,p},$$
 leading to a contradiction. We thus have $\cA_{\gamma,p}^*=0$, which finishes the proof.
 \end{proof}
 
 \section{Symmetries for extension problems}\label{sec:symm}
 
 In this section we show that the argument provided in Lemma \ref{lem:real-complex} about real-valuedness of maximizers extends to a more general setting. A similar remark was made independently in \cite{BroOliQui-18}. If $N\ge1$, $S\subset\R^N$, $\sigma$ is a Borel measure on $S$, and $f\in L^1(S,\sigma)$, we define its Fourier transform as
 $$\forall x\in\R^N,\,\check{f}(x):=\int_S e^{ix\cdot\omega}f(\omega)\,d\sigma(\omega).$$
 Previously, we considered the case $N=2$, $S=\{(\xi,\xi^3),\,\xi\in\R\}$ the cubic curve and the measure $\sigma$ being the push-forward of the measure $|\xi|^\gamma d\xi$ through the map 
 $\xi\in\R\mapsto(\xi,\xi^3)\in S$. Notice that in the optimization problem \eqref{eq:best-airy}, the $L^2$-norm is taken with respect to another measure on $S$ than $\sigma$. As a consequence, let $\sigma'$ be another Borel measure on $S$, and $q\ge2$. Define
 $$\cM(S,\sigma,\sigma',q):=\sup\left\{\frac{\norm{\check{f}}_{L^q(\R^N)}}{\norm{f}_{L^2(S,\sigma')}},\ f\in L^1(S,\sigma)\cap L^2(S,\sigma')\setminus\{0\}\right\}$$
 and, under the symmetry assumption $S=-S$, also its `symmetric' version
 $$\cM_{\text{sym}}(S,\sigma,\sigma',q):=\sup\left\{\frac{\norm{\check{f}}_{L^q(\R^N)}}{\norm{f}_{L^2(S,\sigma')}},\ f\in L^1(S,\sigma)\cap L^2(S,\sigma')\setminus\{0\},\ f(\omega)=\bar{f(-\omega)}\ \text{a.e. on}\ S\right\}.$$
 We then have the following statement.

 \begin{lemma}
   Let $N\ge1$, $S\subset\R^N$ and $\sigma$, $\sigma'$ Borel measures on $S$. Assume that $(S,\sigma)$ and $(S,\sigma')$ are symmetric with respect to the origin, that is, $S=-S$ and $\sigma(A)=\sigma(-A)$, $\sigma'(A)=\sigma'(-A)$ for all $A$ Borel subset of $S$. Then, for any $q\ge2$, we have 
   $$\cM(S,\sigma,\sigma',q)=\cM_{\text{sym}}(S,\sigma,\sigma',q).$$
Moreover, there is an optimizer for $\cM(S,\sigma,\sigma',q)$ if and only if there is one for $\cM_{\text{sym}}(S,\sigma,\sigma',q)$.
 \end{lemma}

We emphasize that in the definition of $\cM(S,\sigma,\sigma',q)$ and $\cM_{\text{sym}}(S,\sigma,\sigma',q)$ we impose the condition $f\in L^1(S,\sigma)$ only in order to have $\check f$ a priori well-defined. Once it is shown that $\cM(S,\sigma,\sigma',q)<\infty$ it follows that $\check f\in L^q(\R^N)$ for any $f\in L^2(S,\sigma')$ and the condition $f\in L^1(S,\sigma)$ can be dropped. In particular, for an optimizer for $\cM(S,\sigma,\sigma',q)$ or $\cM_{\text{sym}}(S,\sigma,\sigma',q)$ we do not require this condition.
 
 \begin{remark}
  This result applies to $S=\Sph^{N-1}$ and $\sigma=\sigma'$ is the standard surface measure on $\Sph^{N-1}$, which is the case of the Stein--Tomas theorem. 
 \end{remark}

 \begin{proof}
  Since the inequality $\ge$ is trivial, let us prove the inequality $\le$. Let $f\in L^1(S,\sigma)\cap L^2(S,\sigma')$, $f\neq0$. We split $f=f_1+if_2$ where 
  $$f_1(\omega):=\frac{f(\omega)+\bar{f(-\omega)}}{2},\quad f_2(\omega):=\frac{f(\omega)-\bar{f(-\omega)}}{2i},$$
  so that $f_1(-\omega)=\bar{f_1(\omega)}$ and $f_2(-\omega)=\bar{f_2(\omega)}$ for a.e. $\omega\in S$. Using the symmetry of $(S,\sigma)$, we deduce that $\check{f_1}$ and $\check{f_2}$ are real-valued, so that $|\check{f}|^2=|\check{f_1}|^2+|\check{f_2}|^2$ and hence by the triangle inequality in $L^{q/2}(\R^N)$
  \begin{align*}
   \norm{\check{f}}_{L^q(\R^N)} &= \norm{|\check{f_1}|^2+|\check{f_2}|^2}_{L^{q/2}(\R^N)}^{1/2}\\
   &\le \left(\norm{\check{f_1}}_{L^q(\R^N)}^2+\norm{\check{f_2}}_{L^q(\R^N)}^2\right)^{1/2}\\
   &\le \cM_{\text{sym}}(S,\sigma,\sigma',q)\left(\norm{f_1}_{L^2(S,\sigma')}^2+\norm{f_2}_{L^2(S,\sigma')}^2\right)^{1/2}.
  \end{align*}
We notice now that $|f_1(\omega)|^2+|f_2(\omega)|^2=(1/2) (|f(\omega)|^2 + |f(-\omega)|^2)$ for all $\omega\in S$, and therefore, by symmetry of $S$ and $\sigma'$,
$$
\norm{f_1}_{L^2(S,\sigma')}^2+\norm{f_2}_{L^2(S,\sigma')}^2 = \|f\|_{L^2(S,\sigma')}^2 \,.
$$
By taking the supremum over all $f$, we therefore obtain the inequality $\le$ in the lemma. Moreover, if $f$ is an optimizer and the suprema are finite, then tracking the case of equality shows that either $f_1$ or $f_2$ is also a maximizer, showing the desired property.
 \end{proof}

 \begin{remark}
  The previous proof clearly extends to mixed Lebesgue spaces (as in the case of the cubic curve), as long as the Lebesgue exponents are greater than $2$.
 \end{remark}

%
%

\begin{thebibliography}{53}

\bibitem{Allaire-92}
{\sc G.~Allaire}, {\em Homogenization and two-scale convergence}, SIAM J. Math.
  Anal., 23 (1992), pp.~1482--1518.

\bibitem{Aubin-76}
{\sc T.~Aubin}, {\em \'equations diff\'erentielles non lin\'eaires et
  probl\`eme de {Y}amabe concernant la courbure scalaire}, J. Math. Pures Appl.
  (9), 55 (1976), pp.~269--296.

\bibitem{BegVar-07}
{\sc P.~B{\'e}gout and A.~Vargas}, {\em Mass concentration phenomena for the
  {$L^2$}-critical nonlinear {S}chr\"odinger equation}, Trans. Amer. Math.
  Soc., 359 (2007), pp.~5257--5282.

\bibitem{BenBezCarHun-09}
{\sc J.~Bennett, N.~Bez, A.~Carbery, and D.~Hundertmark}, {\em Heat-flow
  monotonicity of {S}trichartz norms}, Anal. PDE, 2 (2009), pp.~147--158.

\bibitem{Bourgain-98}
{\sc J.~Bourgain}, {\em Refinements of {S}trichartz' inequality and
  applications to {$2$}{D}-{NLS} with critical nonlinearity}, Internat. Math.
  Res. Notices,  (1998), pp.~253--283.

\bibitem{BreLie-83}
{\sc H.~Br{\'e}zis and E.~H. Lieb}, {\em A relation between pointwise
  convergence of functions and convergence of functionals}, Proceedings of the
  American Mathematical Society, 88 (1983), pp.~486--490.

\bibitem{BreNir-83}
{\sc H.~Br{\'e}zis and L.~Nirenberg}, {\em Positive solutions of nonlinear
  elliptic equations involving critical {S}obolev exponents}, Comm. Pure Appl.
  Math., 36 (1983), pp.~437--477.

\bibitem{BroOliQui-18}
{\sc G.~Brocchi, D.~Oliveira~e Silva, and R.~Quilodr{\'a}n}, {\em Sharp
  {S}trichartz inequalities for fractional and higher order {S}chr{\"o}dinger
  equations}.

\bibitem{CarKer-07}
{\sc R.~Carles and S.~Keraani}, {\em On the role of quadratic oscillations in
  nonlinear {S}chr\"odinger equations. {II}. {T}he {$L^2$}-critical case},
  Trans. Amer. Math. Soc., 359 (2007), pp.~33--62 (electronic).

\bibitem{Carneiro-09}
{\sc E.~Carneiro}, {\em A sharp inequality for the {S}trichartz norm}, Int.
  Math. Res. Not. IMRN,  (2009), pp.~3127--3145.

\bibitem{CarFosOliThi-17}
{\sc E.~Carneiro, D.~Foschi, D.~Oliveria~e Silva, and C.~Thiele}, {\em A sharp
  trilinear inequality related to fourier restriction on the circle}, Rev. Mat.
  Iberoam., 33 (2017), pp.~1463--1486.

\bibitem{CarOliSou-17}
{\sc E.~Carneiro, D.~Oliveira~e Silva, and M.~Sousa}, {\em Extremizers for
  {F}ourier restriction on hyperboloids}, arXiv preprint arXiv:1708.03826,
  (2017).

\bibitem{ChrColTao-03}
{\sc M.~Christ, J.~Colliander, and T.~Tao}, {\em Asymptotics, frequency
  modulation, and low regularity ill-posedness for canonical defocusing
  equations}, Amer. J. Math., 125 (2003), pp.~1235--1293.

\bibitem{ChrSha-12a}
{\sc M.~Christ and S.~Shao}, {\em Existence of extremals for a {F}ourier
  restriction inequality}, Anal. PDE, 5 (2012), pp.~261--312.

\bibitem{EkeTem-99}
{\sc I.~Ekeland and R.~T\'emam}, {\em Convex analysis and variational
  problems}, vol.~28 of Classics in Applied Mathematics, Society for Industrial
  and Applied Mathematics (SIAM), Philadelphia, PA, english~ed., 1999.
\newblock Translated from the French.

\bibitem{Foschi-07}
{\sc D.~Foschi}, {\em Maximizers for the {S}trichartz inequality}, J. Eur.
  Math. Soc. (JEMS), 9 (2007), pp.~739--774.

\bibitem{Foschi-15}
\leavevmode\vrule height 2pt depth -1.6pt width 23pt, {\em Global maximizers
  for the sphere adjoint {F}ourier restriction inequality}, J. Funct. Anal.,
  268 (2015), pp.~690--702.

\bibitem{FraLie-12}
{\sc R.~L. Frank and E.~H. Lieb}, {\em Sharp constants in several inequalities
  on the {H}eisenberg group}, Ann. of Math. (2), 176 (2012), pp.~349--381.

\bibitem{FraLie-15}
\leavevmode\vrule height 2pt depth -1.6pt width 23pt, {\em A compactness lemma
  and its application to the existence of minimizers for the liquid drop
  model}, SIAM J. Math. Anal., 47 (2015), pp.~4436--4450.

\bibitem{FraLieSab-16}
{\sc R.~L. Frank, E.~H. Lieb, and J.~Sabin}, {\em Maximizers for the
  {S}tein-{T}omas inequality}, Geom. Funct. Anal., 26 (2016), pp.~1095--1134.

\bibitem{GerMeyOru-97}
{\sc P.~G{\'e}rard, Y.~Meyer, and F.~Oru}, {\em In\'egalit\'es de {S}obolev
  pr\'ecis\'ees}, in S\'eminaire sur les \'Equations aux {D}\'eriv\'ees
  {P}artielles, 1996--1997, \'Ecole Polytech., Palaiseau, 1997, pp.~Exp.\ No.\
  IV, 11.

\bibitem{GinVel-79a}
{\sc J.~Ginibre and G.~Velo}, {\em On a class of nonlinear schr{\"o}dinger
  equations. i. the cauchy problem, general case}, Journal of Functional
  Analysis, 32 (1979), pp.~1--32.

\bibitem{Goncalves-17}
{\sc F.~Gon{\c{c}}alves}, {\em Orthogonal polynomials and sharp estimates for
  the schr{\"o}dinger equation}, International Mathematics Research Notices,
  (2017).

\bibitem{Grafakos-book-08}
{\sc L.~Grafakos}, {\em Classical {F}ourier analysis}, vol.~249 of Graduate
  Texts in Mathematics, Springer, New York, second~ed., 2008.

\bibitem{HunSha-12}
{\sc D.~Hundertmark and S.~Shao}, {\em Analyticity of extremizers to the
  {A}iry-{S}trichartz inequality}, Bull. Lond. Math. Soc., 44 (2012),
  pp.~336--352.

\bibitem{HunZha-06}
{\sc D.~Hundertmark and V.~Zharnitsky}, {\em On sharp {S}trichartz inequalities
  in low dimensions}, Int. Math. Res. Not.,  (2006), pp.~Art. ID 34080, 18.

\bibitem{JiaShaSto-14}
{\sc J.~Jiang, S.~Shao, and B.~Stovall}, {\em Linear profile decompositions for
  a family of fourth order {S}chr\"odinger equations}, arXiv preprint
  arXiv:1410.7520,  (2014).

\bibitem{JiaPauSha-10}
{\sc J.-C. Jiang, B.~Pausader, and S.~Shao}, {\em The linear profile
  decomposition for the fourth order {S}chr\"odinger equation}, J. Differential
  Equations, 249 (2010), pp.~2521--2547.

\bibitem{KeeTao-98}
{\sc M.~Keel and T.~Tao}, {\em Endpoint {S}trichartz estimates}, Amer. J.
  Math., 120 (1998), pp.~955--980.

\bibitem{KenPonVeg-91}
{\sc C.~E. Kenig, G.~Ponce, and L.~Vega}, {\em Oscillatory integrals and
  regularity of dispersive equations}, Indiana Univ. Math. J., 40 (1991),
  pp.~33--69.

\bibitem{KenPonVeg-00}
\leavevmode\vrule height 2pt depth -1.6pt width 23pt, {\em On the concentration
  of blow up solutions for the generalized {K}d{V} equation critical in
  {$L^2$}}, in Nonlinear wave equations ({P}rovidence, {RI}, 1998), vol.~263 of
  Contemp. Math., Amer. Math. Soc., Providence, RI, 2000, pp.~131--156.

\bibitem{KilShaVis-12}
{\sc R.~Killip, S.~Kwon, S.~Shao, and M.~Visan}, {\em On the mass-critical
  generalized {K}d{V} equation}, Discrete Contin. Dyn. Syst., 32 (2012),
  pp.~191--221.

\bibitem{KilVis-book}
{\sc R.~Killip and M.~Visan}, {\em Nonlinear {S}chr{\"o}dinger equations at
  critical regularity}, Evolution equations, 17 (2013), pp.~325--437.

\bibitem{Kunze-03}
{\sc M.~Kunze}, {\em On the existence of a maximizer for the {S}trichartz
  inequality}, Commun. Math. Phys., 243 (2003), pp.~137--162.

\bibitem{Lieb-83b}
{\sc E.~H. Lieb}, {\em Sharp constants in the {H}ardy--{L}ittlewood--{S}obolev
  and related inequalities}, Ann. Math., 118 (1983), pp.~349--374.

\bibitem{LieLos-01}
{\sc E.~H. Lieb and M.~Loss}, {\em Analysis}, vol.~14 of Graduate Studies in
  Mathematics, American Mathematical Society, Providence, RI, second~ed., 2001.

\bibitem{Lions-84}
{\sc P.-L. Lions}, {\em The concentration-compactness principle in the calculus
  of variations. {T}he locally compact case, {P}art {I}}, Ann. Inst. Henri
  Poincar{\'e}, 1 (1984), pp.~109--149.

\bibitem{MerVeg-98}
{\sc F.~Merle and L.~Vega}, {\em Compactness at blow-up time for {$L^2$}
  solutions of the critical nonlinear {S}chr\"odinger equation in 2{D}},
  Internat. Math. Res. Notices,  (1998), pp.~399--425.

\bibitem{MoyVarVeg-99}
{\sc A.~Moyua, A.~Vargas, and L.~Vega}, {\em Restriction theorems and maximal
  operators related to oscillatory integrals in {$\mathbb R^3$}}, Duke Math.
  J., 96 (1999), pp.~547--574.

\bibitem{Oliveira-14}
{\sc D.~Oliveira~e Silva}, {\em Extremizers for {F}ourier restriction
  inequalities: convex arcs}, J. Anal. Math., 124 (2014), pp.~337--385.

\bibitem{Quilodran-13}
{\sc R.~Quilodr\'an}, {\em On extremizing sequences for the adjoint restriction
  inequality on the cone}, J. Lond. Math. Soc. (2), 87 (2013), pp.~223--246.

\bibitem{Quilodran-15}
\leavevmode\vrule height 2pt depth -1.6pt width 23pt, {\em Nonexistence of
  extremals for the adjoint restriction inequality on the hyperboloid}, J.
  Anal. Math., 125 (2015), pp.~37--70.

\bibitem{CarOli-16}
{\sc R.~Quilodr{\'a}n and D.~Oliveira~e Silva}, {\em On extremizers for
  {S}trichartz estimates for higher order {S}chr\"odinger equations}, arXiv
  preprint arXiv:1606.02623,  (2016).

\bibitem{Shao-09b}
{\sc S.~Shao}, {\em The linear profile decomposition for the {A}iry equation
  and the existence of maximizers for the {A}iry {S}trichartz inequality},
  Anal. PDE, 2 (2009), pp.~83--117.

\bibitem{Shao-09}
\leavevmode\vrule height 2pt depth -1.6pt width 23pt, {\em Maximizers for the
  {S}trichartz inequalities and the {S}obolev-{S}trichartz inequalities for the
  {S}chr{\"o}dinger equation}, Electronic J. of Differential Equations,
  (2009), pp.~1--13.

\bibitem{Shao-15}
\leavevmode\vrule height 2pt depth -1.6pt width 23pt, {\em On existence of
  extremizers for the {T}omas-{S}tein inequality for {$S^1$}}, J. Funct. Anal.,
  270 (2016), pp.~3996--4038.

\bibitem{Simon-basic-complex}
{\sc B.~Simon}, {\em Basic complex analysis}, A Comprehensive Course in
  Analysis, Part 2A, American Mathematical Society, Providence, RI, 2015.

\bibitem{Stein-86}
{\sc E.~M. Stein}, {\em Oscillatory integrals in {F}ourier analysis}, in
  Beijing lectures in harmonic analysis ({B}eijing, 1984), vol.~112 of Ann. of
  Math. Stud., Princeton Univ. Press, Princeton, NJ, 1986, pp.~307--355.

\bibitem{Strichartz-77}
{\sc R.~Strichartz}, {\em Restrictions of {F}ourier transforms to quadratic
  surfaces and decay of solutions of wave equations}, Duke Math. J., 44 (1977),
  pp.~705--714.

\bibitem{Tao-03}
{\sc T.~Tao}, {\em A sharp bilinear restriction estimate for paraboloids},
  Geom. Funct. Anal., 13 (2003), pp.~1359--1384.

\bibitem{Tao-07}
\leavevmode\vrule height 2pt depth -1.6pt width 23pt, {\em Two remarks on the
  generalised {K}orteweg-de {V}ries equation}, Discrete Contin. Dyn. Syst., 18
  (2007), pp.~1--14.

\bibitem{Tao-09}
\leavevmode\vrule height 2pt depth -1.6pt width 23pt, {\em A pseudoconformal
  compactification of the nonlinear {S}chr\"odinger equation and applications},
  New York J. Math., 15 (2009), pp.~265--282.

\bibitem{Tomas-75}
{\sc P.~A. Tomas}, {\em A restriction theorem for the {F}ourier transform},
  Bull. Amer. Math. Soc., 81 (1975), pp.~477--478.

\end{thebibliography}

\end{document}